\documentclass[11pt]{amsart}

\usepackage{amsfonts,amsmath,latexsym,amssymb,verbatim,amsbsy,amsthm,esint}

\usepackage{graphicx}

\usepackage{epigraph}
\usepackage[top=0.8in, bottom=0.8in, left=0.8in, right=0.8in]{geometry}

\usepackage[dvipsnames]{xcolor}

\usepackage[colorlinks=true, pdfstartview=FitV, linkcolor=RoyalBlue,citecolor=ForestGreen, urlcolor=blue]{hyperref}

\definecolor{labelkey}{rgb}{0,0,1}

\theoremstyle{plain}
\newtheorem{THEOREM}{Theorem}[section]

\newtheorem{corollary}[THEOREM]{Corollary}
\newtheorem{lemma}[THEOREM]{Lemma}
\newtheorem{proposition}[THEOREM]{Proposition}

\theoremstyle{definition}

\theoremstyle{remark}

\newtheorem{remark}[THEOREM]{Remark}

\newcommand{\lem}[1]{Lemma~\ref{#1}}

\newcommand{\prop}[1]{Proposition~\ref{#1}}

\newcommand{\sect}[1]{Section~\ref{#1}}
\def \a {\alpha}
\def \b {\beta}

\def \d {\delta}
\def \e {\varepsilon}

\def \k {\kappa}
\def \l {\lambda}
\def \n {\nabla}
\def \s {\sigma}
\def \th {\theta}
\def \o {\omega}

\def \G {\Gamma}

\def \O {\Omega}

\def \bkap {{\boldsymbol \kappa}}
\def \bxi {{\boldsymbol \xi}}
\def \bu {{\bf u}}

\def \by {{\bf y}}

\def \cE {\mathcal{E}}
\def \cF {\mathcal{F}}

\def \cI {\mathcal{I}}

\def \cN {\mathcal{N}}

\newcommand{\Z}{\ensuremath{\mathbb{Z}}}   %%% integers
\newcommand{\R}{\ensuremath{\mathbb{R}}}   %%% reals
\newcommand{\T}{\ensuremath{\mathbb{T}}}   %%% torus

\newcommand{\E}{\ensuremath{\mathbb{E}}}
\renewcommand{\S}{\ensuremath{\mathbb{S}}} %%% sphere

\def \lan {\langle}
\def \ran {\rangle}
\def \p {\partial}
\def \ra {\rightarrow}
\def \ss {\subset}
\def \bs {\backslash}

\renewcommand{\geq}{\geqslant}

\renewcommand{\leq}{\leqslant}

\DeclareMathOperator{\supp}{supp} %
\DeclareMathOperator{\diver}{div} %
\def \dx  {\, \mbox{d}x}

\def \dby  {\, \mbox{d}\by}

\def \ds  {\, \mbox{d}s}
\def \dr  {\, \mbox{d}r}
\def \dk  {\, \mbox{d}\kappa}

\def \dbxi  {\, \mbox{d}\bxi}
\def \dmu  {\, \mbox{d}\mu}
\def \dnu  {\, \mbox{d}\nu}
\def \dell  {\, \mbox{d}\ell}

\def \hmin {h_{\mathrm{min}}}
\def \hmax {h_{\mathrm{max}}}

\def \pmin {p_{\mathrm{min}}}
\def \pmax {p_{\mathrm{max}}}

\begin{document}

\title{Volumetric theory of intermittency in fully developed turbulence}

\author{Alexey Cheskidov}

\address{Department of Mathematics, Statistics, and Computer Science,
	University of Illinois at Chicago, and School of Mathematics, Institute for Advanced Study, Princeton.}
\email{acheskid@uic.edu}

\author{Roman Shvydkoy}

\address{Department of Mathematics, Statistics, and Computer Science,
	University of Illinois at Chicago, and Isaac Newton Institute for Mathematical Sciences, Cambridge.}

\email{shvydkoy@uic.edu}

\date{\today}



\thanks{\textbf{Acknowledgment.}  
	The work of R. Shvydkoy was supported in part by NSF
	grant DMS-2107956. He thanks Isaac Newton Institute for Mathematical Sciences for hospitality and support. The work of A. Cheskidov was partially supported by the NSF grant DMS-1909849 and Charles Simonyi Endowment at the Institute for Advanced Study.}

\begin{abstract} This study introduces a new family of volumetric flatness factors which give a rigorous parametric description of the phenomenon of intermittency in fully developed turbulent flows. These quantities gather information about the most ``active" part of a velocity field at each scale $\ell$, and allows one to define a dimension function $p \to D_p$ that recovers intermittency correction to the structure exponents $\zeta_p$  in an explicit way. In particular, the predictions of the Frisch-Parisi multifractal formalism can be recovered in a systematic and rigorous way.

Within this framework we identify active regions that carry the most energetic part of a velocity field at a given scale $\ell$. A threshold for what constitutes to be  active is defined explicitly. Active regions have proven to be experimentally observable in our previous joint work \cite{Ph-paper}, and shown to capture concentration of the energy cascade as $\ell \to 0$, in \cite{CS2014}.

We present several examples of fields which exhibit arbitrary multifractal spectrum within theoretically permitted limitations. At the same time we demonstrate with the use of a probabilistic argument that a random field is expected to produce the classical K41 spectrum in the limit $\ell\to 0$. Intermittent deviations from K41 theory are estimated at any finite scale also.  

Lastly, we present a detailed information-theoretic analysis of the introduced objects. In particular, we quantify concentration of a given source-field in terms of the volume factors, thresholds, and active regions.  
\end{abstract}

%\epigraph{"Sometimes the answer is not a theorem, and a theorem is not the answer."}{Charlie Doering}

\maketitle

\tableofcontents

\section{Introduction}

The notion of a structure function goes back to the seminal work of Kolmogorov \cite{Kolmogorov}. For a length scale $\ell$, the $n$th order structure function is defined is
\begin{equation} \label{eq:def_S_p}
S_p(\ell) =  \lan |\d_\ell \bu|^p \ran =\lan |\bu(r+\vec{\ell},t) - \bu(r,t)|^p\ran,
\end{equation}
with some appropriate average $\lan \cdot \ran$ where $|\vec{\ell}| =\ell$. It is customary to represent the structure functions by the scaling law
\begin{equation} \label{eq:S_p_exponents-intro}
S_p(\ell) \sim (\varepsilon \ell)^{\frac{p}{3}}\left(\frac{\ell}{L}\right)^{\zeta_p -\frac{p}{3}} \sim U^p  \left(\frac{\ell}{L}\right)^{\zeta_p},
\end{equation}
where $U$ and $L$ are characteristic velocity and length scale, $\zeta_p$ are scaling exponents, and $\varepsilon$ is the average energy dissipation rate
predicted to be independent of the Reynolds number by Kolmogorov's theory of turbulence. More precisely,
\[
\frac{\varepsilon L}{U^3} \sim Re^0,
\]
for large $Re$, which has been proved in one direction (as an upper bound) for long time averages of solutions to the forced 3D Navier-Stokes equations \cite{doering_foias_2002}, while the other direction remains a major open problem in mathematical theory of turbulence.

Kolmogorov's theory of turbulence also predicts self-similar scaling $\d_\ell \bu \sim \ell^{1/3}$ and, as a consequence, the law
%where $\varepsilon$ is the average energy dissipation rate and $\zeta_p$ are scaling exponents, predicted to be
\begin{equation}\label{e:p3}
    \zeta_p = \frac{p}{3}.
\end{equation}
In fact, for the longitudinal third order structure function $S^{\parallel}_3(\ell) = \lan   (\d_\ell \bu \cdot \ell/|\ell| )^3 \ran$ this is known as the Kolmogorov's $\frac45$th law and can be derived from the first principles: $ S^{\parallel}_3(\ell) = - \frac45 \e \ell$, see \cite{Frisch} and \sect{s:45th} below. 
The validity of \eqref{e:p3} has been challenged almost immediately by Landau  due to intermittency -- non-uniform distribution of energy transfer from one scale to another -- although the first credible experimental confirmation of this phenomenon came much later in the work of Anselmet et. al. \cite{anselmet}. We refer to \cite{es-survey, Frisch} for detailed historic surveys of the subject.   While deviations from the law \eqref{e:p3} have now being supported by numerous data, see \cite{PhysRevE.48.R33, PhysRevLett.77.1488, Nature12466, Ph-paper, PhysRevFluids.5.054605} for the latest, the exact formula for $\zeta_p$ remains unknown.

%The energy dissipation rate $\varepsilon$, defined as $\epsilon = \nu \lan |\nabla \bu|^2 \ran$, according to zeroth law, is expected to be %independent of the Reynolds number. More precisely,
%\[
%\frac{\varepsilon L}{U^3} \sim Re^0,
%\]
%for large $Re$, where $L$ and $U$ are characteristic lengthscale and velocity respectively. Remarkably, this law has been proved in one direction (as an upper bound) for long time averages of solutions to the forced 3D Navier-Stokes equations \cite{}. Obtaining a lower bound for the energy dissipation rate remains a major open problem in mathematical theory of turbulence. In this paper we omit $\varepsilon$ from \eqref{eq:def_S_p} effectively changing the units of $\ell$, which is quite common in the literature, especially since non-Kolmogorov scaling of $\zeta_p$ would require cumbersome correctors in \eqref{eq:def_S_p} to keep track of the units.

Several phenomenological models have been invented to account for intermittency. At the core of these models lies the notion of a {\em spatial intermittency dimension} $D$ at scale $\ell$, a number between $0$ and $3$, such that 
\[
\text{Number of active eddies} \sim \left(\frac{L}{\ell}\right)^{D}.
\]
Here $L$ is the size of the domain, and by ``active eddies" we understand fluid blobs of size $\ell$ of the filtered field that participate in the energy transfer to the next scale $\ell/2$. The case $D=3$  corresponds to Kolmogorov's regime
where the whole region is actively involved in the transfer at each scale in the inertial range. The other borderline 
case is $D=0$ (extreme intermittency), where the number of eddies involved is bounded over all the scales. Intermittent flows, where $D<3$, exhibit deviations from Kolmogorov's regime by about $10\%$ experimentally, see the literature above. If defined in rigorous mathematical terms, the notion of intermittency can be useful in formulating effective regularity criteria for solutions to the classical incompressible fluid models, such as the Navier-Stokes equation. In \cite{CS2014}   the definition was proposed based on the Littlewood-Paley decomposition.  More precisely, localizing the fluid velocity $\bu$ in frequency by taking the Littlewood-Paley projection $\bu_\l$ to the dyadic shell of radius $\lambda = \ell^{-1}$, eddies can be defined as active regions where $\bu_\l$ is large is some $L^p$ norm. Then counting the number of eddies at each scale leads to the dimension as a saturation parameter $D_{q,p}$ in the classical Bernstein inequality
\begin{equation} \label{eq:Bernstein-intro}
\| \bu_\l \|_{L^q} = \l^{(3 - D_{q,p}) \left( \frac1p - \frac1q \right) } \|\bu_\l\|_{L^p}.
\end{equation}
This notion of intermittency when applied to Leray-Hopf solutions of the Navier-Stokes equations, with appropriate temporal averaging, leads to numerous rigorous results \cite{CS2014,CS-unified,10.1017/prm.2018.33,2112.11606}. For instance, $D_{\infty,2} \geq 3/2$ implies regularity of solutions, \cite{CS-unified}. In a parallel series of works, a similar in spirit theory based on the concept of sparseness was developed by Grujic et. al. in \cite{Zoran2018,Zoran2019}. 

While \eqref{eq:Bernstein-intro} is suitable for mathematical analysis, in physics a well-accepted measure of intermittency is the multifractal spectrum (MFR) introduced by Frisch an Parisi \cite{Frisch-Parisi} -- the dimension $d_h(\ell)$ of a set where the H\"older exponent $h$ is attained by a fluid flow at scale $\ell$. In this paper we introduce a new analytic framework based on a suitably defined volumetric quantities associated with the field $\d_{\ell} \bu$ that reconciles the two approaches in a rigorous way. The main idea is based on viewing velocity increments $\d_{\ell} \bu$ as a source of information whose concentration in active regions $A_p$ is measured by explicitly defined dimensions $D_p$ and then translated into  intermittency corrections to the structure functions $S_p$.
%. Explicitly defined dimensions $D_p$ measure concentration of information in active regions $A_p$ and translate into  intermittency corrections to the structure functions $S_p$.
Moreover, we show that $D_p$ coincides with the MFR spectrum $d_h$ under the transformation $h=\zeta_p'$. % On the other hand, the volumetric framework allows us to perform rigorous analysis and hence enhance mathematical understanding of intermittency. 

To describe the main ingredients of the framework more specifically let us pass to the adimensional variables $\bu \to \bu/U$ and $\ell \to \ell/L$ and consider the scailing exponents defined according to \eqref{eq:S_p_exponents-intro}
\[
\zeta_p: = \log_{\ell} \lan |\d_\ell \bu|^p \ran,
\]
where $\lan |\d_\ell \bu|^p \ran$ is the appropriate average (e.g., in space and angles) of the velocity increments at length scale $\ell$, even though the precise nature of the localisation of $\bu$ is not important for many of our results. As a function of $p$, the scaling exponent $\zeta_p$ is smooth, concave, with $\zeta_0=0$.
Then for Lebesgue exponents $p,q\in \mathbb{R}$ we 
define  the \emph{$(q,p)$-volume factor} and dimension $D_{q,p}$ at scale $\ell$, 
\begin{equation}\label{e:vol-intro-beginning}
V_{q,p}(\ell) := \frac{ \lan |\d_\ell \bu|^q \ran^{\frac{p}{p-q}}}{ \lan |\d_\ell \bu|^p \ran^{\frac{q}{p-q}}}, \qquad 
 D_{q,p} (\ell) := 3 - \log_\ell V_{q,p},
\end{equation}
consistent with \eqref{eq:Bernstein-intro}.
Remarkably, in the limit, the intermittency dimension $D_p:=\lim_{q\to p} D_{q,p}$  turns out to be 
\[
D_p=3- \zeta_p +p \zeta'_p,
\]
and hence the dimension $D_p$ coincides with the MFR dimension $d_h$ defined via Legendre transform
\begin{equation}\label{e:MFR-intro}
D_{p(h)} = d_h  :=\inf_p(3+ hp-\zeta_p).
\end{equation}
Indeed, $D_p=d_h$ at the point where the infimum is attained, i.e., $h=\zeta_p'= \frac{\lan |\d_\ell \bu|^p \log_\ell |\d_\ell \bu| \ran}{\lan |\d_\ell \bu|^p \ran}$. The resulting map $p \to D_p$ reproduces the same multifractal spectrum as $h \to d_h$ but now in terms of the same Lebesgue variable $p$ that parameterises the structure exponents $\zeta_p$. This allows us to derive an explicit intermittency correction to the classical linear profile \eqref{e:p3}:
\[
\zeta_p = p\zeta_0'  + \mathcal{I}_p, \quad \text{ where } \quad \cI_p = p \int_0^p \frac{D_s-3}{s^2} \ds.
\]
In Subsection~\ref{s:Functional_properties} we describe functional properties of $\zeta_p$, $D_p$, and $d_h$. In particular, all possible scenarios are classified at two endpoints $(p_{\max} , h_{\min} )$ and $(p_{\min} ,h_{\max} )$. 

The dimension $D_p$ is naturally connected to the active volume 
\[
V_p:= \lim_{q \to p} V_{q,p} = \ell^{3-D_p},
\]
which is an upper bound on the size of the active region $A$ that can be parameterized in terms of the H\"older or Lebesgue exponents as 
\[
A^{\text{H\"older}}_h:=\{x: c \ell^h \leq |\d_\ell \bu(x)| \leq  C \ell^h \} = A^{\text{Lebesgue}}_p:=\{x: c \ell^{\zeta_p'} \leq |\d_\ell \bu(x)| \leq  C \ell^{\zeta_p'} \}.
\]
Here, the two exponents are in the same correspondence as before $h=\zeta'_p $. Indeed, we show that
\begin{equation}\label{e:AVp}
    \lan 1_{A} \ran \lesssim \ell^{3-d_h}=  \ell^{3-D_p}=  V_p.
\end{equation}
For more detailed multifractal analysis of H\"older spectrum based on Besov scaling we refer to \cite{J1997,J2000, 2111.03493}. 
%\[
%\lan 1_{A^{\text{H\"older}}_h} \ran \lesssim \ell^{3-d_h}=  V_{p(h)}, \qquad \lan 1_{A^{\text{Lebesgue}}_p} \ran %\lesssim \ell^{3-D_p}=  V_p.
%\]

In principle, an arbitrary smooth concave exponent $\zeta_p$ passing through the origin could be achieved by a velocity field as we demonstrate in some concrete examples in Subsections~\ref{s:beta} and \ref{s:example}. However, in fully developed turbulence one might expect universal laws for $\zeta_p$ and intermittency dimensions $D_p$, $d_h$.  Mathematically, such universality can be studied by analysing vector fields with randomized Fourier coefficients. To this end, in Subsection~\ref{s:random_field} we compute the expected spectrum for a random vector field normalized in accordance with the Kolmogorov $\frac45$th law: $\E \zeta_3 = 1$. We derive the following bounds for the expected value of $\zeta_p$:
\[
 \frac{p}{3}+ \frac{-\ln \sqrt{\pi} + p\ln 2 + \ln \G( (p+1) / 2 )}{\ln \ell}  \leq \E  \zeta_p \leq  p \left(\frac{1}{3} - \frac{2\ln 2}{3 \ln \ell} \right).
\]
Intermittency dimension $D_p$ for the lower bound is computed to be
\[
D_p = 3 - \frac{-\ln \sqrt{\pi} +\ln \G( (p+1) / 2 ) - \frac{p}{2}\psi((p+1) / 2)}{\ln \ell},
\]
where $\psi(z) = \G'( z )/\G( z )$ is the polygamma function. While randomized fields may exhibit intermittency at finite scales, the classical K41 spectrum holds in the limit as $\ell \to 0$:
\[
\lim_{\ell \to 0} \E \zeta_p = \frac{p}{3},
\]
confirming the heuristics that frequency randomized fields loose concentration and eddies start occupying the whole space. In the particular case of $p=2$, the scaling of the second order structure function becomes $S_2(\ell) \to \ell^\frac23$, as $\ell \to 0$. This can be translated directly to asymptotic behavior of the energy spectrum $E(\k) \to \frac{c_0}{\k^{5/3}}$, as $\k \to \infty$, recovering the classical law of $5/3$-rds. Relationship between the energy spectrum and second structure function will be established for any scaling laws in \prop{p:spec}.

In Section~\ref{s:active} we give an information-theoretic analysis of the field based on the introduce volumetric parameters. The main main idea of this study is to establish 

1) how much information about the vector field $\d_\ell \bu$ is contained in sets of probability measure $V_{p,q}$ and $V_p$;

2) how much information about the vector field $\d_\ell \bu$ is contained in the corresponding sets $A_{q,p}$ and $A_p$ themselves, while their probability measures being asymptotically subordinate to the respective volumes in a manner similar to \eqref{e:AVp}.

One surprising result of this study shows that the dyadic volume factors of type $V_{q,q/2}$ collect most information about the $L^q$-source $|\d_\ell \bu|^q/ \lan  |\d_\ell \bu|^p \ran$ among all the factors $V_{q,p}$, $p<q$. A particular case is given by the classical flatness factor $\cF = V_{4,2}^{-1}$. At the same time, the 1-parameter families $V_p$ and $A_p$'s collect information about the entropy source $F \ln F$, where $F =  |\d_\ell \bu|^p / \lan  |\d_\ell \bu|^p \ran$.

What constitutes and defines the active part of the field $\d_\ell \bu$ can be expressed via an explicit family of threshold speeds:
\begin{equation}\label{}
s_{q,p} = \frac{ \lan |\d_\ell \bu|^q \ran^{\frac{1}{q-p}}}{ \lan |\d_\ell \bu|^p \ran^{\frac{1}{q-p}}},
\end{equation}
and their limiting 1-parameter counterpart
\begin{equation}\label{e:spdefintro}
\lim_{q \ra p} s_{q,p} = s_p=  \exp\left\{  \frac{\lan |\d_\ell \bu|^p \ln |\d_\ell \bu| \ran}{\lan |\d_\ell \bu|^p \ran}\right\}.
\end{equation}
Information about these parameters can be gathered experimentally. Results of this research were reported in  \cite{Ph-paper}.

\section{Intermittency and multifractal spectrum} \label{s:intermittency_and_MFR}

\subsection{Intermittency dimensions $D_p, d_h$ and scaling exponents $\zeta_p$}

For each scale $\ell$ and $p\in \R$ we define
\[
\zeta_p: = \zeta_p(\ell)= \log_{\ell} \lan |\d_\ell \bu|^p \ran.
\]
In other words, $\zeta_p$ is the exponent of the $p$th order structure function,
\begin{equation} \label{eq:zeta_p_d_l}
S_p(\ell) =  \lan |\d_\ell \bu|^p \ran =\ell^{\zeta_p}, \qquad   p \in \mathbb{R}.
\end{equation}
Note that all the exponents pass through the origin,
\[
\zeta_0 = 0.
\]
Although in general $\zeta$'s depend on $\ell$, we will suppress it to simplify the notation.

 Let us define the domain of $\zeta_p$: 
\[
p_{\mathrm{min}} = \inf\{p\leq 0: \zeta_p \text{ is finite}\}, \qquad p_{\mathrm{max}} = \sup\{p\geq 0: \zeta_p\text{ is finite}\}.
\]
Note that the endpoints $p_{\min}$ or $p_{\max}$ might not be be in the domain of $\zeta_p$.
\begin{lemma} \label{l:e:zeta_p}
For $\ell \leq 1$, as a function of $p$, $\zeta_p$ is concave and smooth on $(p_{\mathrm{min}}, p_{\mathrm{max}})$ with
\begin{equation} \label{e:zeta'}
 \zeta'_p = \frac{\lan |\d_\ell \bu|^p \log_\ell |\d_\ell \bu| \ran}{\lan |\d_\ell \bu|^p \ran}.
\end{equation}
\end{lemma}
\begin{proof} The smoothness follows directly from the definition. To show concavity, first note that $\zeta_p$ is trivially concave in $p$,
\[
\zeta_{\th p_1 + (1-\th)p_2} \geq \th \zeta_{p_1} + (1-\th) \zeta_{p_2}, \quad 0 \leq \theta \leq 1,
\]
by the H\"older inequality on each interval $[0,p_{\mathrm{max}})$ and $(p_{\mathrm{min}},0]$ separately. Combining with the fact that it is also differentiable at $p=0$ shows that $\zeta_p$ is concave globally on the entire interval. 
\end{proof}

The basic idea of the Frisch-Parisi theory consists of relating the scaling exponents $\zeta_p$ to a range of parameters $d_h(\ell)$, called multifractal spectrum (MFR for short), that at a given scale $\ell$ measure the dimension of a set $A_h$ where the field $\bu$ achieves the H\"older exponent $h$. The relation between the MFR and the scaling exponents can be derived heuristically based on an incidence argument, see \cite{Frisch} and is given by Legendre transform 
\begin{equation}\label{e:MFR}
  d_h(\ell)  =\inf_p(3+ hp-\zeta_p).
\end{equation}
Here the infimum is taken over the domain of $\zeta_p$. We will use \eqref{e:MFR} as the definition of $d_h(\ell)$ over domain
\[
[h_{\min}, h_{\max}] = [\inf_p \zeta'_p,\sup_p \zeta'_p],
\]
where again the infimum and supremum are taken over the domain of $\zeta_p$.

We can in fact justify \eqref{e:MFR} rigorously defining  the following active region capturing specific H\"older regularity of order $h$ at a specific scale $\ell$
\begin{equation} \label{e:Ah}
A_h= \{x: c \ell^h \leq |\d_\ell \bu(x)| \leq  C \ell^h \}.
\end{equation}
Here $C,c>0$ are fixed %adimensional 
constants. Then, by Chebyshev's inequality, for $p \geq 0$ we obtain
\begin{equation}
\lan 1_{A_h} \ran \leq c^{-p} \ell^{-hp}\lan |\d_\ell \bu|^p \ran,
\end{equation}
and for $p<0$, using the upper bound defining $A_h$,
\begin{equation}
\lan 1_{A_h} \ran \leq C^{-p} \ell^{-hp}\lan |\d_\ell \bu|^p \ran.
\end{equation}
Combining with the very definition \eqref{eq:zeta_p_d_l} we obtain
\[
\lan 1_{A_h} \ran \lesssim \ell^{-hp+\zeta_p}, \qquad p \in \R.
\]
Since $\ell$ is small, the optimal bound is obtained at the point of the minimal exponent resulting in 
\begin{equation}\label{e:Ahdh}
\lan 1_{A_h} \ran \lesssim \ell^{-\inf_p(hp-\zeta_p)} =\ell^{3-d_h(\ell)}.
\end{equation}

This rigorous inequality demonstrates that the asymptotic value $d_h = \lim_{\ell \to 0} d_h(\ell)$ captures the fractal dimension of the set of H\"older regularity $h$.

Knowing the exponents $\zeta_p$ or, equivalently, structure functions $S_p$, allows us to compute the MFR according to \eqref{e:MFR}. On the other hand, assuming that we knew the MFR in the first place, its 
practical use can come from inverting formula \eqref{e:MFR} over the domain of $\zeta_p$
\begin{equation}\label{e:MFRinv}
    \zeta_p = \inf_h(3+ ph - d_h(\ell)),
\end{equation}
where the infimum is taken over $[h_{\mathrm{min}},h_{\mathrm{max}}]$.
%We should keep in mind that this inversion continuously extends $\zeta_p$ to $[p_{\min}, p_{\max}]$ including the endpoints where $\zeta_p$ might be infinite, so it should only be used on the domain of $\zeta_p$.
This gives a direct access to computing scaling exponents $\zeta_p$ from the multifractal spectrum.  It is therefore highly desirable to define the dimensions $d_h(\ell)$ on the first place and in a way that is amenable to analysis.

To this end, we start by defining a family of quantities called volume factors.  Let $-\infty \leq p,q \leq \infty$, $q\neq p$. We define  the \emph{$(q,p)$-volume factor} at scale $\ell$  as follows
\begin{equation}\label{e:vol-intro}
V_{q,p}(\ell) = \frac{ \lan |\d_\ell \bu|^q \ran^{\frac{p}{p-q}}}{ \lan |\d_\ell \bu|^p \ran^{\frac{q}{p-q}}}.
\end{equation}
We will postpone the detailed study of these parameters till later sections. In short, the factors capture the probability measure of the set where ``most'' of the $L^q$-mass of the field is concentrated, see Lemma \ref{l:AV}. As $\ell\to 0$ such ``active'' sets settle on a fractal of certain dimension, which is determined by the exponential rate of decay of the volume factors. So, we further define
\begin{equation} \label{e:Dqp}
    D_{q,p} (\ell) = 3 - \log_\ell V_{q,p}.
\end{equation}
These dimensions can be directly expressed in terms of scaling exponents $\zeta_p$'s via
\begin{equation}\label{e:Dpq2}
D_{q,p}(\ell) = 3-\frac{p\zeta_q - q\zeta_p}{p-q}.
\end{equation}
\begin{remark} A more illuminating role of the dimensions is seen if one rewrites its  definition in the following form
\begin{equation}\label{e:Bern1}
 \lan |\d_\ell \bu|^p \ran^{\frac{1}{p}} = \ell^{(3-D_{p,q})\left(\frac{1}{p}-\frac{1}{q}\right)} \lan |\d_\ell \bu|^q \ran^{\frac{1}{q}}.
\end{equation}
This equality bears direct relation with the Bernstein equality for Littlewood-Paley projections of the field $\bu$, see \cite{Grafakos}. In fact if we replace the velocity increments $\d_\ell \bu$  with the corresponding projections $\bu_k$ to the dyadic shells in the Fourier space around a wavenumber $\l_k = 2^k/L$, where $L$ is the characteristic length scale of the periodic domain $L\T^n$, then the classical Bernstein inequality states that for $p<q$ one has 
\[
\| \bu_k \|_q \leq c \l_k^{3 \left( \frac1p - \frac1q \right) } \|\bu_k\|_p,
\]
where $c>0$ is an adimensional constant.  At the same time, on the periodic domain $\| \bu_k \|_q  \geq c(L) \|\bu_k\|_p$. So, there exists a parameter $D_{q,p}$ which gauges the level of saturation of the inequality:  
\begin{equation}\label{e:Bern2}
\| \bu_k \|_q = \l_k^{(3 - D_{q,p}) \left( \frac1p - \frac1q \right) } \|\bu_k\|_p.
\end{equation}
We can see that \eqref{e:Bern2} is precisely \eqref{e:Bern1} taken with the corresponding spacial scale $\ell  = \l_k^{-1}$.  

 This definition of dimension based on the Bernstein inequality for the particular case $p=2$, $q = 3$ was adopted by the authors in work \cite{CS2014}, which recovered the $\b$-model of intermittency \cite{FSN-beta}, see also \cite{Frisch}. We will come back to this point in \sect{s:beta} and present a similar interpretation based on the finite difference approach taken here. 
\end{remark}

Now, letting $q \to p$ we define a new 1-parameter family of volumes and corresponding dimensions:
\begin{equation}
    \begin{split}
        V_{q,p} \to V_p &= \lan |\d_{\ell} \bu|^p \ran \exp\left\{ - \frac{\lan |\d_{\ell} \bu|^p \ln |\d_{\ell} \bu|^p \ran}{\lan |\d_{\ell} \bu|^p \ran}\right\} \\
        D_{q,p} \to D_p &= 3 - \log_\ell \lan |\d_\ell \bu|^p \ran  +  \frac{\lan |\d_\ell \bu|^p \log_\ell |\d_\ell \bu|^p \ran}{\lan |\d_\ell \bu|^p \ran}.
    \end{split}
\end{equation}
Notice that
\begin{equation} \label{eq:V_p=l^(3-D_p)}
V_p  = \ell^{3-D_p},
\end{equation}
and by virtue of the formula \eqref{e:Dpq2} we obtain the following relationship between the new dimensions $D_p$ and $\zeta_p$'s: 
\begin{equation} \label{eq:dzetadp}
D_p=3- \zeta_p +p \zeta'_p. 
\end{equation}
Recalling the classical fact that the infimum \eqref{e:MFR} occurs at a point $p$ where $h = \zeta'_p$, we conclude that 
\[
D_p(\ell) = d_h(\ell), \quad h = \zeta'_p.
\]

This formula shows the new and old dimensions are in natural correspondence via the monotone change of variables $h = \zeta'_p$. Consequently the range $D_p$ covers the same full multifractal spectrum as $p$ runs through the interval $[p_{\mathrm{min}}, p_{\mathrm{max}}]$, just as $d_h$ does as $h$ runs through its interval of possible values $[h_{\mathrm{min}},h_{\mathrm{max}}]$, with a possible exclusion of the endpoints. Since $\zeta'_p$ is  monotonely decreasing in $p$, the end-point values of $h = \zeta'_p$ are determined by
\begin{equation}\label{e:hends}
h_{\mathrm{min}}=  \zeta'_{p_{\mathrm{max}}} \qquad  \text{and} \qquad  h_{\mathrm{max}}=  \zeta'_{p_{\mathrm{min}}},
\end{equation}
where the values are to be interpreted as limits which always exist, even if infinite, by monotonicity.

Thanks to \eqref{eq:V_p=l^(3-D_p)}, the bound \eqref{e:Ahdh} on the active region $A_h$  can now be expressed in terms of the volume $V_p$:
\begin{equation}
\lan 1_{A_h} \ran \lesssim \ell^{3-d_h}= \ell^{3-D_p} = V_p,
\end{equation}
and hence we shall refer to $V_p$ as an active volume.
Also, the threshold $\ell^h$ determining the active region $A_h$ can be expressed more explicitly using that $h = \zeta_p'$ and formula \eqref{e:zeta'}
\begin{equation}\label{e:spfirst}
\ell^h = \ell^\frac{\lan |\d_\ell \bu|^p \log_\ell |\d_\ell \bu| \ran}{\lan |\d_\ell \bu|^p \ran}= \exp \left\{ \frac{\lan |\d_\ell \bu|^p \ln |\d_\ell \bu| \ran}{\lan |\d_\ell \bu|^p \ran} \right\}= :s_p.
\end{equation}
These quantities $s_p$ determine another 1-parameter family of values which determine ``active thresholds" intimately tied with the other volumetric quantities introduced above. We will address this systematically  in  Section~\ref{s:active}.

By analogy with the relation \eqref{e:MFRinv} we can now use $D_p$ as a new determining family of dimensions. In fact solving the simple ODE \eqref{eq:dzetadp} for $\zeta_p$ we obtain, for all $p\in (p_{\mathrm{min}}, p_{\mathrm{max}})$,
\begin{equation}\label{e:zetaDp}
    \zeta_p = p \zeta_0' + p \int_0^p \frac{D_s-3}{s^2} \ds.
\end{equation}
So the relationship between $\zeta_p$'s and $D_p$'s is purely differential, given by \eqref{eq:dzetadp}, \eqref{e:zetaDp}, while the relationship between $\zeta_p$'s and the old dimensions $d_h$ is variational \eqref{e:MFR}, \eqref{e:MFRinv}.

Formula \eqref{e:zetaDp} shows explicitly the intermittency correction to the classical linear profile expressed in terms of the new dimensions
\begin{equation}\label{e:Ip}
    \cI_p = p \int_0^p \frac{D_s-3}{s^2} \ds.
\end{equation}
In particular, if $D_p \equiv 3$ as in the Kolmogorov regime, then $\cI_p = 0$ and we recover the classical theory of non-intermittent turbulence.

Interestingly, while $D_p = D_{p,p}$ define only a particular subfamily dimensions $D_{q,p}$, the full 2-parameter family can be restored from its diagonal values $D_p$. Indeed, plugging \eqref{e:zetaDp} into \eqref{e:Dpq2} we obtain
\begin{equation}
    D_{q,p} = 3 + \frac{qp}{q-p} \int_p^q \frac{D_s - 3}{s^2} \ds.
\end{equation}
This can be further simplified if $p$ and $q$ have the same sign by splitting the integral:
\begin{equation}
    D_{q,p} = \frac{qp}{q-p} \int_p^q \frac{D_s}{s^2} \ds, \quad qp >0.
\end{equation}

\subsection{Functional properties of dimensions} \label{s:Functional_properties} Now that we have defined the multifractal spectrum, it is desirable to understand the range of possible H\"older exponents $[h_{\mathrm{min}},h_{\mathrm{max}}]$ and other graph properties of the dimensions.

First, directly from the definition it is clear that $D_p$ is smooth on $(p_{\mathrm{min}}, p_{\mathrm{max}})$. A more explicit formula for $D_p$ is given by
\begin{equation}\label{e:Dpu}
D_p = 3 - \log_\ell \lan |\d_\ell \bu|^p \ran  +  \frac{\lan |\d_\ell \bu|^p \log_\ell |\d_\ell \bu|^p \ran}{\lan |\d_\ell \bu|^p \ran}.
\end{equation}
In particular, $D_0 = 3$. We also have
\[
 D'_p = p \zeta''_p,
\]
which in view of concavity, $ \zeta''_p \leq 0$, implies that $D'_p \geq 0$ for $p\leq 0$, and $D'_p \leq 0$ for $p \geq 0$.  Consequently, $D_p$ attains its  maximum at the origin.

As a function of $p$ the dimensions $D_p$ are generally not concave, see Figure~\ref{f:case4}. However, as a function of H\"older exponent $h$, the dimensions $d_h$ are obviously concave from the general properties of the Legendre transform. In fact, if $p=p(h)$ is the point where the infimum in \eqref{e:MFR} is attained, which is the inverse of the decreasing function $h = \zeta'_p$, then $p(h)$ is decreasing as well. Moreover, we have
\[
d'_h = p(h), \quad d''_h = p'(h) = \frac{1}{\zeta''_p} \leq 0, \quad h\in(\hmin,\hmax),
\]
but this might not hold at the endpoints $h_{\min}$ or $h_{\max}$ as we will see below. 
We also see that negative $p$'s define the interval of decreasing behavior for $d_h$ and positive values define the interval of the increase. In terms of $h$, this means that $d_h$ is increasing on  the interval $(h_{\min}, \zeta'_0]$, picks at value $d_{\zeta'_0} =  D_0 = 3$, and decreases on $[\zeta'_0, h_{\max})$, see Figure~\ref{f:case1}.

\begin{figure}
\centering
	\includegraphics[width=5in]{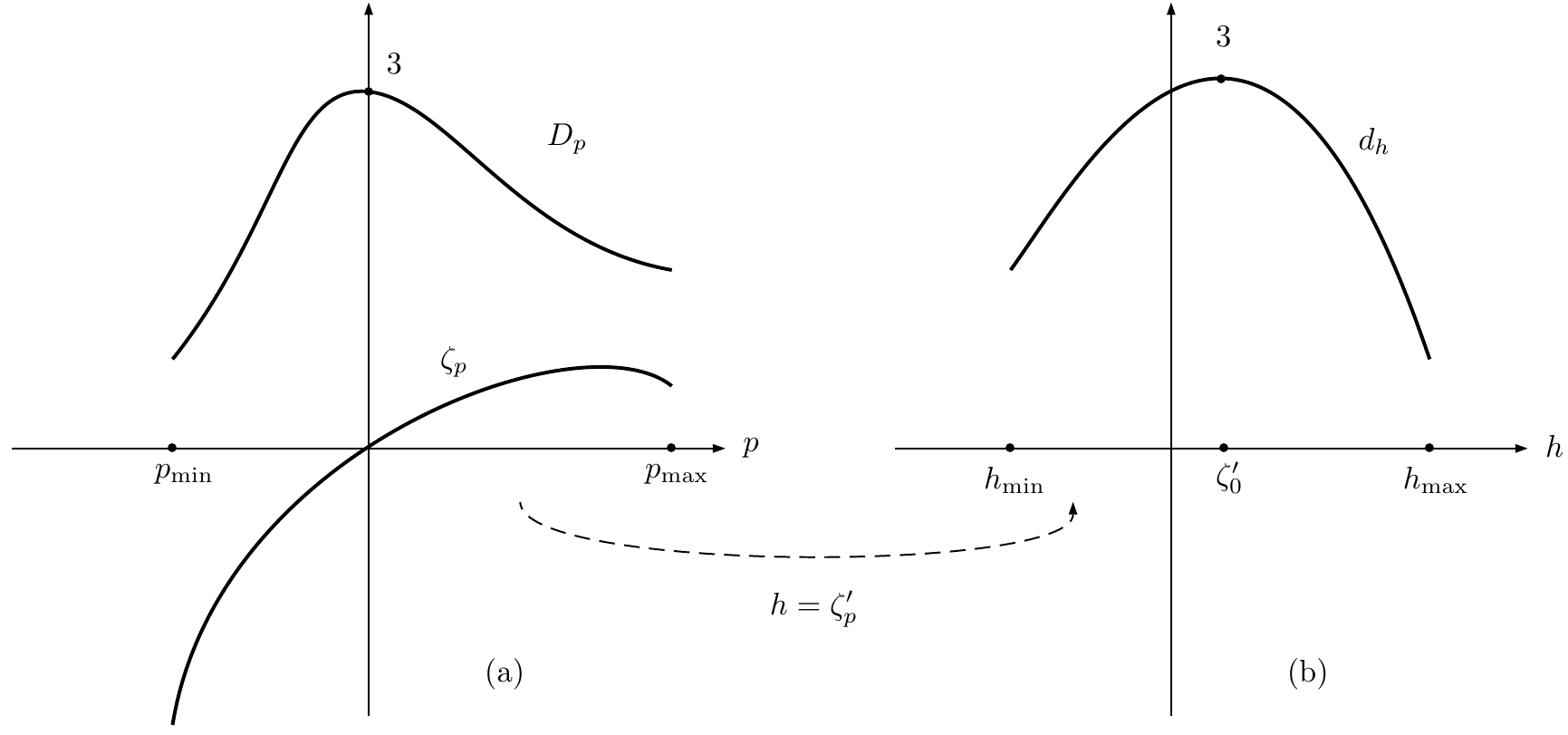}
	\caption{(a) generic graphs of $D_p$ and $\zeta_p$; (b) generic graph of $d_h$. This also corresponds to case (1) of the end-point behavior. }
	\label{f:case1}
\end{figure}

\begin{figure}
\centering
	\includegraphics[width=5in]{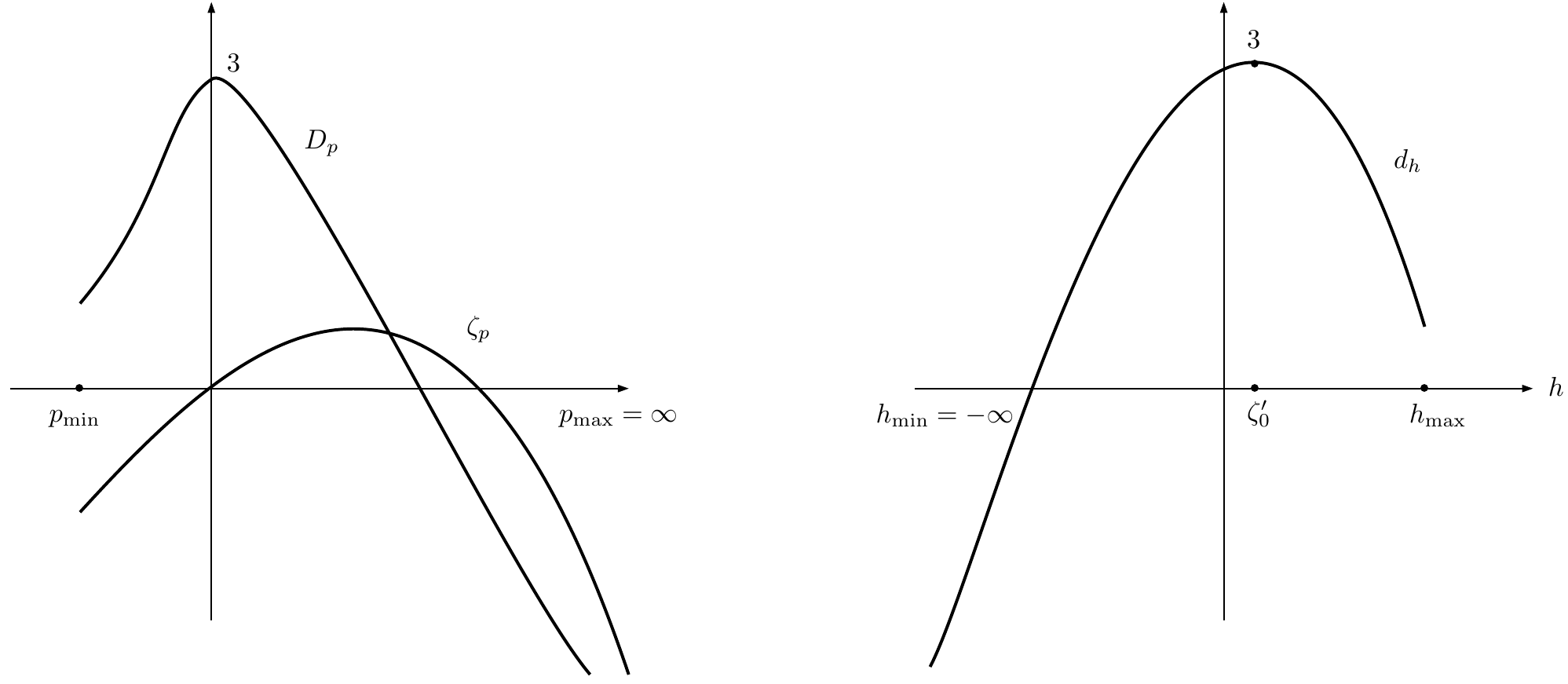}
	\caption{Case (2) of the end-point behavior. }
	\label{f:case2}
\end{figure}

\begin{figure}
\centering
	\includegraphics[width=5in]{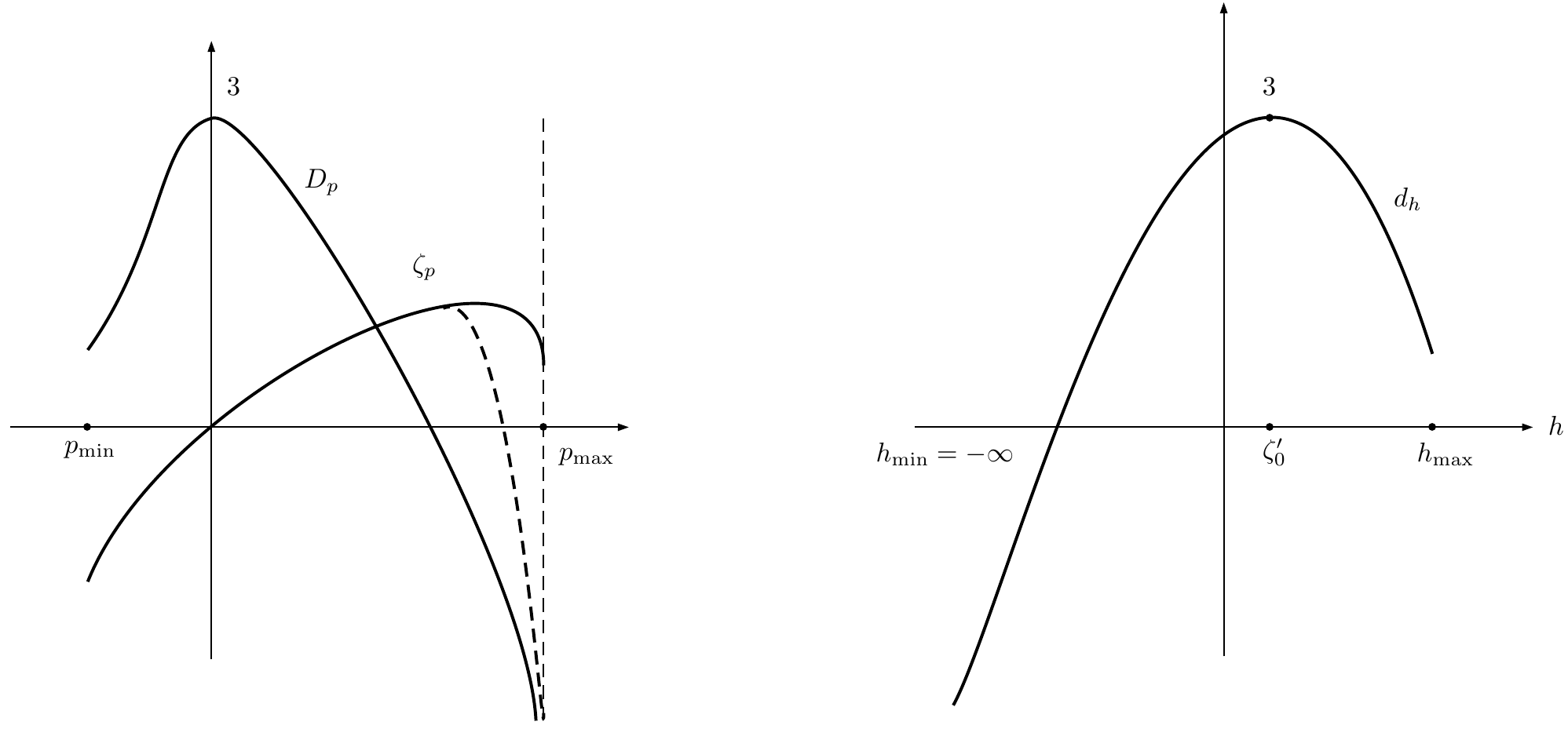}
	\caption{Case (3) of the end-point behavior. }
	\label{f:case3}
\end{figure}

\begin{figure}
\centering
	\includegraphics[width=5in]{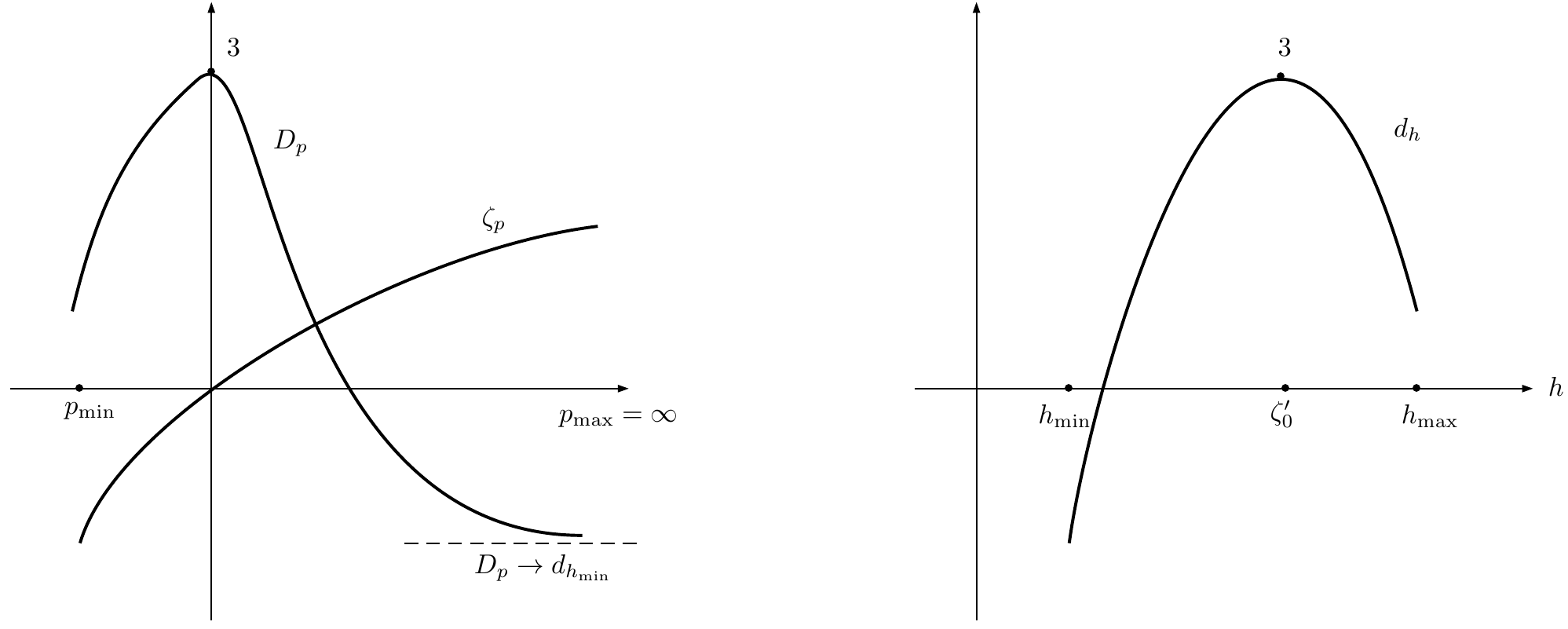}
	\caption{Case (4) of the end-point behavior portraying the situation when $d_{h_{\min}}$ is finite. Note that in this case $D_p$ has the horizontal asymptote at $d_{h_{\min}}$. }
	\label{f:case4}
\end{figure}

\begin{figure}
\centering
	\includegraphics[width=5in]{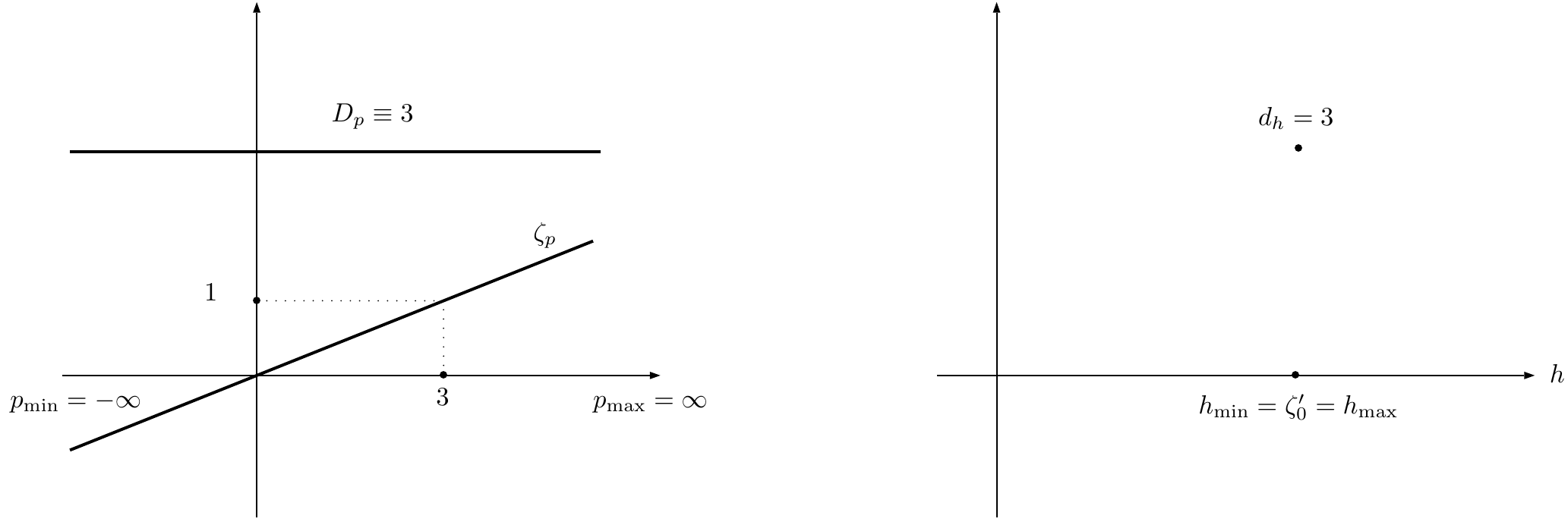}
	\caption{Case (5) represents the classical Kolmogorov-41 regime of turbulence.  }
	\label{f:case5}
\end{figure}
%Note that at the point $p=p(h)$,
%\[
%d'_h = p, \qquad \zeta_p' = h,
%\]
%for all $p \in (p_{\min}, p_{\max})$ and $h \in (h_{\min}, h_{\max})$. However, this might not be true at the endpoints.

Now we turn to the analysis of two endpoints $(p,h)=(p_{\max} , h_{\min} )$ and $(p,h)=(p_{\min} ,h_{\max} )$. For each of them there are five possible scenarios. We will only describe them for $(p_{\max} , h_{\min} )$ as for the other end point the corresponding cases can be obtained by interchanging $\min$ and $\max$.
\begin{enumerate}

    \item $p_{\max}$ and $h_{\min}$ are finite. Then $\zeta'_{p_{\max}}=h_{\min}$, $\zeta_{p_{\max}}$, $d_{h_{\min}}$, and $d'_{h_{\min}}=p_{\max}$ are all finite, see Figure~\ref{f:case1}.
    
    \item $p_{\max}= \infty$ and $h_{\min} =-\infty$. In this case $\zeta'_{p_{\max}}=-\infty$, $\zeta_{p_{\max}} = - \infty$, and $d_{h_{\min}} =d'_{h_{\min}}= -\infty$, see Figure~\ref{f:case2}.
    
    \item $p_{\max}$ is finite but $h_{\min} =-\infty$. In this case $\zeta'_{p_{\max}}=-\infty$, but $d'_{h_{\min}}=p_{\max}$ is finite. There are no restrictions on  $\zeta_{p_{\max}} \in [-\infty, \infty)$ and $d_{h_{\min}} \in [-\infty, 3]$. If $p_{\max}>0$, then $d_{h_{\min}} = - \infty$. See Figure~\ref{f:case3}.

    \item $p_{\max}= \infty$, $h_{\min}$ is finite, but $d'_{\hmin}=\infty$. In this case $\zeta'_{p_{\max}}=h_{\min}$ is finite. There are no restrictions on $\zeta_{p_{\max}}\in[-\infty, \infty]$ and $d_{h_{\min}} \in [-\infty, 3]$, see Figure~\ref{f:case4}.

\item $p_{\max}= \infty$ but $h_{\min}$ and $d'_{\hmin}$ are finite. This implies K41 regime as we prove below. More precisely, $\hmin=\hmax=\zeta_0'$, $\zeta_p = \zeta_0' p$, $D_p=3$ for all $p \in \mathbb{R}$, and $d_{\zeta_0'} =3$. Moreover, if $\zeta_3=1$, then $\zeta_0' = 1/3$. See see Figure~\ref{f:case5}.

\end{enumerate}

Remarkably, the so called {\em linearization phenomenon} occurs in case (4) as $\zeta_p$ converges to an asymptote  $\zeta'_{p_{\max}}=h_{\min}$, which is consistent with some experimental observations and numerical computations \cite{Lashermes2004LimitationOS, faller_geneste_chaabo}.

\begin{lemma}
Suppose that $p_{\max}= \infty$, $h_{\min}>-\infty$ and $d'_{\hmin}<\infty$. Then the conclusions of case (5) follow - the fluid is in the classical K41 regime.
\end{lemma}
\begin{proof} Indeed, in this case $d_{h_{\min}}$ has to be finite, and
%and $D_p$ becomes constant at large $p$ has a horizontal asymptote, 
then the exponents $\zeta_p$ become linear in $p$, $\zeta_p = 3+ p \hmin - d_{\hmin}$ starting from $p \geq d'_{\hmin}$. This is because the infimum runs out of domain and is attained at the end-point $h = \hmin$. 
Also, $D_p = d_{h_{\min}}$ for all $p \geq d'_{\hmin}$ in this case.
%Indeed, for $p\geq d'_{h_{\min}}$ we have, by definition of $D_p$,
%\[
%\begin{split}
%\zeta_p &= 3+ p \zeta_p' - D_p\\
%&=3 +p h_{\min} - d_{h_{\min}}.
%\end{split}
%\]
%Similarly, if $p_{\min}=-\infty$ but $\hmax$ and $d'_{\hmax}$ are finite, then $\zeta_p = 3+ p \hmax - d_{\hmax}$ for all $p \leq d'_{\hmax}$.
If such perfect linearization of $\zeta_p$ indeed occurs, then $\zeta_p''=0$ for $p \geq d'_{\hmin}$. Then
\[
\begin{split}
\zeta''_p =\frac{1}{\ln \ell} \left[\frac{\lan |\d_\ell \bu|^p (\ln |\d_\ell \bu|)^2 \ran}{ \lan |\d_\ell \bu|^p \ran}  - \frac{\lan |\d_\ell \bu|^p \ln |\d_\ell \bu| \ran^2}{ \lan |\d_\ell \bu|^p \ran^2}  \right] =0.
\end{split}
\]
Hence
\[
\lan |\d_\ell \bu|^{p/2} \ln |\d_\ell \bu| |\d_\ell \bu|^{p/2} \ran^2 = \lan |\d_\ell \bu|^p (\ln |\d_\ell \bu|)^2 \ran \lan |\d_\ell \bu|^p \ran,
\]
and the H\"older inequality becomes equality, in which case either $\d_\ell \bu \equiv 0$ or $|\d_\ell \bu|^{p/2} \ln |\d_\ell \bu|$ is proportional to 
$|\d_\ell \bu|^{p/2}$, and hence $|\d_\ell \bu|=constant$. Then $\zeta''_p=0$ for all $p \in \mathbb{R}$, which immediately implies the conclusions of (5).
\end{proof}

\subsection{Finer description of the H\"older spectrum}

In addition to the defining formulas \eqref{e:hends} it is useful to relate the H\'older exponents to the values of $\zeta_p$'s themselves.
\begin{lemma}\label{l:hminmax}
We have the following general inequalities
\begin{equation}\label{e:hratios}
 \hmin \leq  \frac{\zeta_q-\zeta_p}{q-p} \leq \hmax, \quad \forall p, q\in ( \pmin,\pmax).
 \end{equation}
Furthermore, if $\pmax = \infty$ or $\zeta_{\pmax} = - \infty$, then 
\begin{equation}\label{e:hmin=}
\hmin = \lim_{p \to \pmax} \frac{\zeta_p}{p}.
\end{equation}
Respectively, if $\pmin = - \infty$ or $\zeta_{\pmin} = - \infty$, then 
\begin{equation}\label{e:hmax=}
\hmax = \lim_{p \to \pmin} \frac{\zeta_p}{p}.
\end{equation}
\end{lemma}
\begin{proof} To prove \eqref{e:hratios} simply note that
\[
\frac{\zeta_q-\zeta_p}{q-p} = \frac{1}{q-p}\int_p^q \zeta'_s \ds.
\]
The integral, by monotonicity of $\zeta'_s$, is bounded from above and below by $\hmax$ and $\hmin$, respectively.

To prove \eqref{e:hmin=} by concavity and simple integration we have
\begin{equation}\label{e:zetaave}
\zeta'_p \leq \frac{\zeta_p}{p} = \frac{1}{p} \int_0^p \zeta'_s \ds, \qquad \forall p\in [0,\pmax).
\end{equation}
If $\pmax =\infty$, then passing to the limit in \eqref{e:zetaave} we can see that both sides squeeze to the same limit $\zeta'_\infty = \hmin$.  If $\pmax <\infty$ but $\zeta_{\pmax} = - \infty$, this means that both sides converge to $ \hmin = - \infty$. 

The proof at the opposite end goes entirely similar.
\end{proof}

By taking $q=0$ and $p=\pmax$ or $p=\pmin$, respectively, we in general obtain
inequalities 
\[
\hmin \leq  \lim_{p \to \pmax} \frac{\zeta_p}{p}, \qquad \hmax \geq \lim_{p \to \pmin} \frac{\zeta_p}{p}.
\]
A pair of exact identities with dimensional corrections can be obtained by solving for $\zeta'_p = h$ in \eqref{eq:dzetadp}:
\[
\hmin = \lim_{p \to p_{\max}}\left(\frac{ \zeta_p}{p} - \frac{3-D_p}{p} \right), \qquad \hmax = \lim_{p \to p_{\min}}\left(\frac{ \zeta_p}{p} - \frac{3-D_p}{p} \right),
\]
One can use these to derive further finer properties of the end-point behavior of the spectra. For example, in case (4), i.e. if  $p_{\max} = \infty$ and $\hmin$ is finite, there are two possibilities in terms of $d_{\hmin}$ being either finite or infinite. If it is finite, as depicted on Figure~\ref{f:case4}, then clearly $D_p$ has to have a horizontal asymptote at $D_p \to d_{\hmin}$. However, even if $d_{\hmin} = \infty$, the graph of $D_p$ would still have a zero slope at infinity. Indeed, in this case 
\[
\hmin = \lim_{p \to \infty}\left(\frac{ \zeta_p}{p} - \frac{3-D_p}{p} \right) = \hmin - \lim_{p \to \infty} \frac{ D_p}{p} ,
\]
and hence
\[
\lim_{p \to \infty} \frac{ D_p}{p} =0.
\]

In those cases where we do have the exact relations \eqref{e:hmin=} and \eqref{e:hmax=}, the full H\"older spectrum of the field can also be recovered through the intermittency correction formula \eqref{e:zetaDp}. We have
\[
\begin{split}
\hmin & = \zeta_0' + \int_0^{\pmax} \frac{D_s-3}{s^2} \ds, \\
\hmax & = \zeta_0' - \int_{\pmin}^0 \frac{D_s-3}{s^2} \ds.
\end{split}
\]
Consequently, the full width of the spectrum is given by the total intermittency correction
\[
\hmax - \hmin =  \int_{\pmin}^{\pmax} \frac{3 - D_s}{s^2} \ds.
\]

Formulas \eqref{e:hmin=} and \eqref{e:hmax=} also allow to express the end-point H\"older exponents in terms of the maximum  and minimum of the field itself.
\begin{corollary}\label{c:endpoints}
If $\pmax = \infty$, then the minimal H\"older exponent is given by
\[
\hmin = \log_{\ell} \|\d_\ell \bu \|_\infty.
\]
If $\pmin = - \infty$, then the maximal  H\"older exponent is given by
\[
\hmax = \log_\ell \min |\d_\ell \bu|.
\]
\end{corollary}

In particular, if the solution belongs to the largest scaling invariant space $\bu \in B^{-1}_{\infty,\infty}$, then $| \d_\ell \bu| \lesssim \ell^{-1}$, and so we have a limitation on the minimal H\"older exponent
\[
h_{\min} \geq -1.
\]

Looking at the other end, it is natural for exponent to stretch all the way to $\hmax = \infty$ because the flow at any given scale may have regions of infinite smoothness or even analyticity. The likelihood of this scenario is seen from the fact that fluctuations have mean zero, $\lan \d_\ell \bu \ran = 0$. So, for example, for parallel shear flows it necessarily implies that $\d_\ell \bu$ must vanish somewhere, and hence $\hmax = \infty$ according to the corollary above.

\subsection{Under the Kolmogorov $\frac45$th law: recovery of the $\b$-model}\label{s:beta}

It is observed almost universally in experiments and simulations that the exponent of the third structure function is $\zeta_3 = 1$. This in part is evidenced by the Kolmogorov $\frac45$th law which can be derived from the first principles, see \cite{Frisch}, 
\begin{equation}\label{e:K45th} 
S^{\parallel}_3(\ell) = \lan   (\d_\ell \bu \cdot \ell/|\ell| )^3 \ran =   - \frac45 \e \ell,
\end{equation}
where $\e$ is the energy dissipation rate per unit mass.  So, assuming that $\zeta_3 = 1$, formula \eqref{e:Dpq2} implies the famous relation between the structural exponents $\zeta_p$'s and the corresponding dimensions under the $\beta$-model formalism (compare with Frisch's  \cite[(8.31)]{Frisch}):
\begin{equation}\label{e:zetabeta}
	\zeta_p = \frac{p}{3} +(3- D_{p,3})\left( 1 - \frac{p}{3} \right).
\end{equation}
One particular case, $p=2$, reads
\begin{equation}
	\zeta_2 = \frac{2}{3} + \frac13(3- D_{2,3}).
\end{equation}

In the situation when the flow is  self-similar, i.e. $\zeta_2(\ell) = \zeta_2$, we obtain the corresponding correction to the energy spectrum (compare with \cite[(8.32)]{Frisch}):
\begin{equation}
	E(\k) \sim \frac{1}{\k^{ \frac{5}{3} + \frac13(3- D_{2,3})}}.
\end{equation}
In fact, such connection between the scaling of the second order structure function and that of the energy spectrum is not entirely elementary, and we will prove this rigorously in Proposition~\ref{p:spec} below.

The dependence of scaling exponent $h = \zeta'_p$ on $p$ can be derived from \eqref{e:zetabeta},
\begin{equation}
	h = \frac13 - \frac{3-D_{p,3}}{3} - \frac{3-p}{3} D'_{p,3}.
\end{equation} 
This recovers formula \cite[(8.29)]{Frisch} obtained under the assumption of $p$-independent dimension. Note that the value $h = \frac13$ is a part of this spectrum in view of Lemma~\ref{l:hminmax}.

To summarize, the entire set of predictions of the $\beta$-model becomes embedded into the formalism based on the two-parameter spectrum $\{D_{p,q}\}$ in the special case  $q = 3$.

Our last observation concerns the content of the H\"older spectrum $I = [h_{\min},h_{\max}]$.   Dividing \eqref{e:zetabeta} by $p$ and  sending $p \to \pmax$ and $\pmin$ we obtain from \lem{l:hminmax},
\begin{equation}
	\begin{split}
h_{\min} &\leq  \frac{D_{\pmax,3} - 2}{3}, \\
h_{\max} &\geq \frac{D_{\pmin,3} - 2}{3}.
	\end{split}
\end{equation} 
It follows from \lem{l:V6} below that $D_{p,3}$ is a decreasing function of $p$. Hence, the entire range 
\[
 \left\lbrace   \frac{D_{p,3} - 2}{3}:  p\in (\pmin,\pmax) \right\rbrace 
\]
must be a part of the spectrum $I$.

Let us illustrate the findings of this section by an example of a mono-fractal turbulent field. Later in \sect{s:example} we will upgrade it to multi-fractal construction.

We fix a scale $\ell$ and consider an arbitrary collection distribution of $N$ cubes $C_i$, $i=1,\dots, N$, of linear size $\ell$ placed at a distance of $\ell$ from each other. Define \[
\bu = \sum_{i=1}^N u_i \chi_{C_i},
\]
where all $|u_i| = U_0$. The cubes occupy a volume of $N \ell^3 = \ell^{3-D}$, so by definition $D=  - \log_\ell N \in[0,3]$. 

Let us compute the structure exponents now. We have
\begin{equation}\label{e:SpU}
    \lan | \d_\ell \bu |^p \ran = N U_0^p \ell^3 = U_0^p \ell^{3-D}.
\end{equation}
On the other hand, under the Kolmogorov $\frac45$th law, we must have
\[
\lan | \d_\ell \bu |^3 \ran \sim \e \ell.
\]
Consequently, $U_0 = \e^{1/3}\ell^{(D-2)/3}$. Plugging into \eqref{e:SpU} we obtain
\[
 \lan | \d_\ell \bu |^p \ran = \e^{p/3}\ell^{3-D + p(D-2)/3}.
\]
So, we arrive at the formula
\[
\zeta_p = 3-D + p(D-2)/3, \qquad p>0.
\]
This is exactly the spectrum of the mono-fractal model given in \eqref{e:zetabeta}. In fact, computing the dimensions we obtain $D_p = D$ for $p>0$ directly from \eqref{eq:dzetadp}, and all $D_{q,p} = D$, $p,q>0$.
So, as we can see, the definitions recover the natural dimension of the field. 

Computing the velocity thresholds from \eqref{e:spfirst} we obtain all $s_p = U_0$, illustrating the fact that $U_0$ stands out as the active velocity of the eddies.

When $D<3$, the graphs of $D_p$ and $\zeta_p$ experience a jump discontinuity at $p =0$, since as always $D_0 = 3$ and $\zeta_0=0$ by definition. This occurs due to idealization we took of the field itself picking only the active eddies, whereas in more realistic situation the field is more likely to be spread around at lower levels. This would result in a rapid smooth transition from level $3$ to $D$ as far as $D_p$ is concerned, and similarly for $\zeta_p$. Consequently, the H\"older spectrum in our idealized example should be viewed as the ray $[h_{\min}, h_{\max}) = [\frac{D-2}{3}, \infty)$.
The MFR dimensions satisfy $d_h = D$ for all $\frac{D-2}{3} \leq  h <\infty$. 

In Kolmogorov's regime $D=3$ the exponent $\zeta_p$ is continuous at the origin and $h_{\min} = h_{\max} = 1/3$. Also $\zeta_0' =1/3$ and consequently $d_{1/3}=D_0=3$. This case is illustrated on Figure~\ref{f:case5}.

For $D<3$ we have $\zeta'_0 = \infty$, so morally speaking we expect $d_\infty =3$ even though $d_\infty$ is not defined. This should be understood in the sense that $d_{h_{\max}}$ is close to $3$ for a smooth approximation of $\zeta_p$.
%Finally, the formulas of Corollary~\ref{c:endpoints} hold as well despite $\pmin$ being finite.

To conclude, the predictions of our theory manifest themselves as expected, albeit in a somewhat degenerate form.

%%%%%%%%%%%%%%%%%%%%%%%%%%%%%%%%%%%%%%%%%%%%%%%%%%%%%%
\subsection{Energy flux and the third order exponent $\zeta_3$}\label{s:45th} The evidence for $\zeta_3=1$  which we postulated in the previous section can be seen in the context of deterministic weak solutions of the Euler equation
\begin{align} \label{ee}
\frac{\partial \bu}{\partial t} + (\bu \cdot \nabla)\bu &= - \nabla p +f, \\
\nabla \cdot \bu &=0. \label{diver-free}
\end{align}

In \cite{dr} Duchon and Robert showed that for every weak solution $u \in L^3(\O_T)$ to the Euler equation one can associate a distribution $\mathcal D_{\bu}$, representing the anomalous energy dissipation, such that
 \begin{equation}\label{DRE}
    \p_t \left( \frac{1}{2} |\bu|^2 \right) + \diver \left( \bu\left(\frac{1}{2} |\bu|^2 + p\right)\right) + \mathcal D_u = f\cdot \bu.
\end{equation}
In fact, $D_{\bu}$ is the limit in the sense of distributions of
\[
D^\ell_{\bu}(x,t) = \frac{1}{4} \int \nabla \Psi^\ell(y)\cdot \d_y \bu |\d_y \bu|^2 \, dy,
\]
for any non-negative compactly supported mollifier $\Psi^\ell(x) = \ell^{-3} \Psi(x/\ell)$, where $\delta_y \bu = \bu(x+y,t)-\bu(x,t)$. Choosing radially symmetric mollifier $\nabla \Psi^\ell(x)= \ell^{-4}\Phi(|x|/\ell)$ supported on $\alpha^{-1} \ell  \leq r \leq \alpha \ell$ we compute
\[
\begin{split}
\varepsilon_\ell := \lan |\mathcal D^\ell_u |\ran &\leq \frac{1}{4\ell^4} \int_{\alpha^{-1}\ell}^{\alpha \ell} |\Phi(r/\ell)| 4\pi r^2\lan  |\d_r \bu|^3 \ran \, dr\\
& \sim  \ell^{-1}\lan  |\d_r \bu|^3 \ran_{r \sim \ell},
\end{split}
\]
where heuristically $\varepsilon_\ell$ represents the energy flux through scale $\ell$. Thus
\[
\lan \zeta_3(r) \ran_{r \sim \ell} \geq 1 + \log_\ell (c \varepsilon_\ell) \to 1, 
\]
as $\ell \to 0$, provided $0< \lim_{\ell \to 0} \varepsilon_\ell < \infty$, as expected for turbulent flows.

On the other hand, a similar computation yields an upper bound for the longitudinal exponent, defined as
\[
\zeta^\parallel _3(\ell)= \log_\ell \lan (\d_\ell \bu \cdot \ell/|\ell| )^3\ran,
\]
assuming some regularity of $\bu$.

\subsection{A simple example of multifractality} \label{s:example}

By analogy with the example presented in \sect{s:beta} we can construct a field whose multifractal spectrum approximates an arbitrary convex dimension $d_h$, and hence, approximates an arbitrary convex power law $\zeta_p$. 

In order not to make it exceedingly complicated we allow the field to vanish as before, and therefore we will adhere to the case of $\pmin=0$. In other words, we assume that $d_h$ is non-decreasing since $p = d_h'$. Let us also assume that the end-points of the spectrum are finite, i.e.,  $\hmin,\hmax, \pmin=0,\pmax \in \R$. In other words, the slopes of the graphs of $\zeta_p$ and $d_h$ are finite at the ends, and $d_{\hmax}=3$.

Let us fix a sequence $\hmin=h_1<h_2<\dots <h_K = \hmax$. Also, fix dimensions $\mathcal{D}_k = d_{h_k}$, which forms a monotonely increasing sequence. We now define a field $\bu$, by considering $K$ families of cubes $C_{i,k}$, $k=1,\dots,K$,  $i=1,\dots,N_k$, or size $\ell$ separated by a distance of at least $\ell$ form each other. We assume $|u_{i,k}|= \ell^{h_k}$, and $N_k = \frac{1}{\ell^{D_k}}$. Define
\[
\bu = \sum_{k=1}^K \sum_{i=1}^{N_k} u_{i,k} \chi_{C_{i,k}}.
\]
Then
\[
\lan |\d_\ell \bu|^p \ran = \sum_{k=1}^K \ell^{3+ph_k - \mathcal{D}_k} \sim \ell^{\zeta^K_p},
\]
where 
\[
\zeta_p^K = \min_k (3+ph_k - \mathcal{D}_k).
\]
So the (smooth) power spectrum for the constructed vector field $\bu$ is close to a polygon $\zeta_p^K$ with nodal points given by
\[
p_k = \frac{\mathcal{D}_{k+1} - \mathcal{D}_k}{h_{k+1} - h_k}, \quad \zeta_{p_k}^K = 3 + \frac{\mathcal{D}_{k+1} h_k - \mathcal{D}_{k} h_{k+1}}{h_{k+1}-h_k}, \quad k =1, \dots K-1,
\]
and $p_K=0$, $\zeta_{p_K}=0$.
The graph of $\zeta_p$ is inscribed under the polygon, whose edges are tangent to $\zeta_p$ at points $p=d'_{h_k}$.

Now we can use formula \eqref{eq:dzetadp} to find the intermittency dimension $D_p$ corresponding to the polygon $\zeta_p^K$, which will approximate the intermittency dimension of the constructed field $\bu$.
The dimension function $D_p$ is not well-defined at the nodal points $p_k$'s, but otherwise is given by
\[
D_p = \sum_{k=1}^{K} \mathcal{D}_k \chi_{(p_{k+1},p_k)}(p).
\]

As we let $\ell \to 0$, these fields will provide arbitrarily close approximation to the given data demonstrating that at least formally there is no functional restrictions on the possible multifractal spectra.

\subsection{Intermittent spectrum of a random field} \label{s:random_field}

In this section we compute the expected spectrum of a vector field with randomized Fourier coefficients, and show that  in the limit of vanishing scale $\ell \to 0$ the statistical laws of such a field approach the classical K41 prediction. Our computation are explicit enough and it allows to read off intermittency corrections at any finite scale.

Let us assume that the fluid domain is the torus $\T^3$. Let us fix a base field
\[
 \bu_0 = \sum_{k\in \Z^3} u_k e^{i k\cdot x}.
\]
We will assume that $\bu_0$ is isotropic, i.e., the $L^p$ norms of the velocity displacement do not depend on the direction of the displacement, and hence the $p$-th order structure function can be expressed as
\begin{equation}\label{e:basefield}
\lan |\d_\ell \bu_0|^p \ran = \fint_{\T^3} |\bu_0(x+\vec \ell)-\bu_0(x) |^p\, dx,    
\end{equation}
for any vector $\vec \ell$ with $|\vec \ell|=\ell$. Such fields are expected and observed in turbulent flows, and  mathematical examples are abundant. For instance, one can consider any radial field $\bu_0 = \bu_0(|x|)$ supported on a coordinate chart of the torus, or a combination of such disjoint fields separated by a distance of at least $\ell$.

Based on the idea that in a fully turbulent flow Fourier coefficients of a field may point in random directions, we fix a set of independent identically distributed mean-zero Rademacher random variables $\{ \th_k \}_{k\in \Z_+}$, $\th_k=\pm 1$, where
\[
\Z^3_+ = \{ (n_1,n_2,n_3): n_1>0\} \cup  \{ (0,n_2,n_3): n_2>0\}  \cup  \{ (0,0,n_3): n_3>0\} .
\]
We extend it to $\Z^3_-  = \Z^3 \backslash (\Z^3_+ \cup \{0\})$, by $\th_{-k} = {\th}_k$.  Define the random field by
\[
 \bu = \sum_{k\in \Z^3} u_k \th_k e^{i k\cdot x},
\]
where $u_{-k} = \bar{u}_k$ and $u_0 = 0$ are fixed. Note that $\bu$ is no longer isotropic, however, for the 2nd order structure function we still have \eqref{e:basefield} because the $L^2$-norms are not random by Parseval's identity.
In other words, the angle averaging in $\lan |\d_\ell \bu|^2 \ran$ can be omitted. Naturally, as before, we will use $\ell$ for both the displacement vector and its magnitude.

Given an isotropic vector filed, we decompose it as
\[
\d_\ell \bu = \sum_{k\in \Z^3} u_k \th_k e^{i k\cdot x} (e^{i k \cdot \ell} - 1) = \d_\ell \bu^+ + \d_\ell \bu^-,
\]
where $\d_\ell \bu^\pm$ contains modes from $\Z^3_\pm$, respectively.

Let us compute expected value of the $p$-th order structure function. We will be primarily interested in the case where $p$ is large, so we assume that $p\geq 3$ (although the computation below can be modified for all $p>0$).
Note that 
\[
\lan | \d_\ell \bu^- |^p \ran =\lan | \overline{\d_\ell \bu^+} |^p \ran =  \lan | \d_\ell \bu^+ |^p \ran.
\]
Hence, by the H\"older inequality,
\[
\lan | \d_\ell \bu |^p \ran = \lan |\d_\ell \bu^+ + \d_\ell \bu^- |^p \ran \leq \lan (|\d_\ell \bu^+|   + | \d_\ell \bu^- |)^p \ran \leq 2^{p-1} \lan (|\d_\ell \bu^+|^p   + | \d_\ell \bu^- |^p) \ran = 2^p\lan | \d_\ell \bu^+ |^p \ran.
\]
Recall that each $\d_\ell \bu^\pm$ is randomized by independent variables. According to the classical Khintchine inequality,
\[
\begin{split}
\E \lan | \d_\ell \bu |^p \ran &\leq 2^p \lan \E | \d_\ell \bu^+ |^p \ran\\
&\leq  2^pB_p \left\langle \Big( \sum_{k\in \Z^3_+} |u_k|^2|e^{i k \cdot \ell} - 1|^2 \Big)^{p/2} \right\rangle_{\mathrm{angle}}\\
&= 2^pB_p  \left\langle \Big( \frac12\sum_{k\in \Z^3} |u_k|^2|e^{i k \cdot \ell} - 1|^2 \Big)^{p/2} \right\rangle_{\mathrm{angle}}\\
&= 2^{p/2}B_p \lan \lan | \d_\ell \bu |^2 \ran^{p/2} \ran_{\mathrm{angle}} =  2^{p/2}C_p \lan \lan | \d_\ell \bu_0 |^2 \ran^{p/2} \ran_{\mathrm{angle}} \\
&= 2^{p/2}B_p  \lan | \d_\ell \bu_0 |^2 \ran^{p/2} =  2^{p/2}C_p  \lan | \d_\ell \bu |^2 \ran^{p/2},
\end{split}
\]
where
\[
B_p = 2^{p/2}\frac{1}{\sqrt{\pi}} \G( (p+1) / 2 ).
\]
So,
\[
\E \lan | \d_\ell \bu |^p \ran   \leq C_p \lan | \d_\ell \bu |^2 \ran^{p/2}, \qquad C_p = 2^{p/2} B_p.
\]
Since $\zeta_2$ is not random, we have
\[
\lan | \d_\ell \bu |^2 \ran^{p/2} = \ell^{\frac{p}{2} \zeta_2 } = \ell^{\frac{p}{2} \E  \zeta_2 }.
\]
At the same time, by the Jensen inequality, since $\ell^x$ is convex, we have
\[
\E \lan | \d_\ell \bu |^p \ran = \E \ell^{\zeta_p} \geq \ell^{\E \zeta_p}.
\]
Thus, we obtain the following bounds on the structure functions:
\[
 \ell^{\frac{p}{2} \E \zeta_2 } \leq \ell^{\E \zeta_p} \leq C_p \ell^{\frac{p}{2} \E \zeta_2 }.
\]
Here, the left hand side follows by concavity. Potentiating,
\begin{equation} \label{eq:zeta_p_lower_bound}
\frac{p}{2} \E \zeta_2 \geq \E \zeta_p \geq \frac{p}{2} \E \zeta_2 + \log_\ell C_p.
\end{equation} 

So, in the limit as $\ell \to 0$ we obtain
\[
\E \zeta_p = \frac{p}{2} \E \zeta_2.
\]
If the original field is normalized in accordance with the Kolmogorov's $\frac45$th law, then $\E \zeta_3 = 1$, and we obtain the classical K41 spectrum
\[
\lim_{\ell \to 0} \E \zeta_p = \frac{p}{3}. 
\]

Going back to \eqref{eq:zeta_p_lower_bound} we can read off intermittency corrections at each scale quite explicitly:
\begin{equation} \label{e:zetaplower}
\frac{p}{2} \E \zeta_2 \geq \E \zeta_p \geq \frac{p}{2} \E \zeta_2 + \frac{-\ln \sqrt{\pi} + p\ln 2 + \ln \G( (p+1) / 2 )}{\ln \ell}.
\end{equation}
Again, under the Kolmogorov normalization $\E \zeta_3 = 1$, computing \eqref{e:zetaplower} at $p=3$ we obtain the estimates
\[
\frac23 - \frac{4\ln 2}{3 \ln \ell} \geq \E \zeta_2 \geq \frac23.
\]
Plugging back into \eqref{e:zetaplower} we obtain
\begin{equation} \label{e:zetaplower2}
p \left(\frac{1}{3} - \frac{2\ln 2}{3 \ln \ell} \right) \geq \E \zeta_p \geq \frac{p}{3}+ \frac{-\ln \sqrt{\pi} + p\ln 2 + \ln \G( (p+1) / 2 )}{\ln \ell}.
\end{equation}

Figure~\ref{f:graph} shows the graphs of the upper and lower bounds on $\E \zeta_p$ in \eqref{e:zetaplower2}, as well as the intermittency dimension $D_p$ for the lower bound in \eqref{e:zetaplower2} computed according to \eqref{eq:dzetadp}
\begin{equation} \label{eq:entreme_D_p}
D_p = 3 - \frac{-\ln \sqrt{\pi} +\ln \G( (p+1) / 2 ) - \frac{p}{2}\psi((p+1) / 2)}{\ln \ell},
\end{equation}
where $\psi(z) = \G'( z )/\G( z )$ is the polygamma function. The smallest H\"older exponent in this case is $h_{\min} = - \infty$. We  also see that the intermittency correction term 
\begin{equation} \label{eq:extreme_I_p}
I_p(\ell) =  \frac{-\ln \sqrt{\pi} + p\ln 2 + \ln \G( (p+1) / 2 )}{\ln \ell} 
\end{equation}
becomes more prominent as $p$ increases with $\ell$ being fixed, and for any fixed $p$ it vanishes as $\ell \to 0$ pushing the exponent to the Kolmogorov regime $\E \zeta_p \to \frac{p}{3}$.

%%%%%%%%%%%%%%%%%%%%%%%%%%%%%%%%%%%%%%%%

\begin{figure}
\centering
	\includegraphics[width=4.5in]{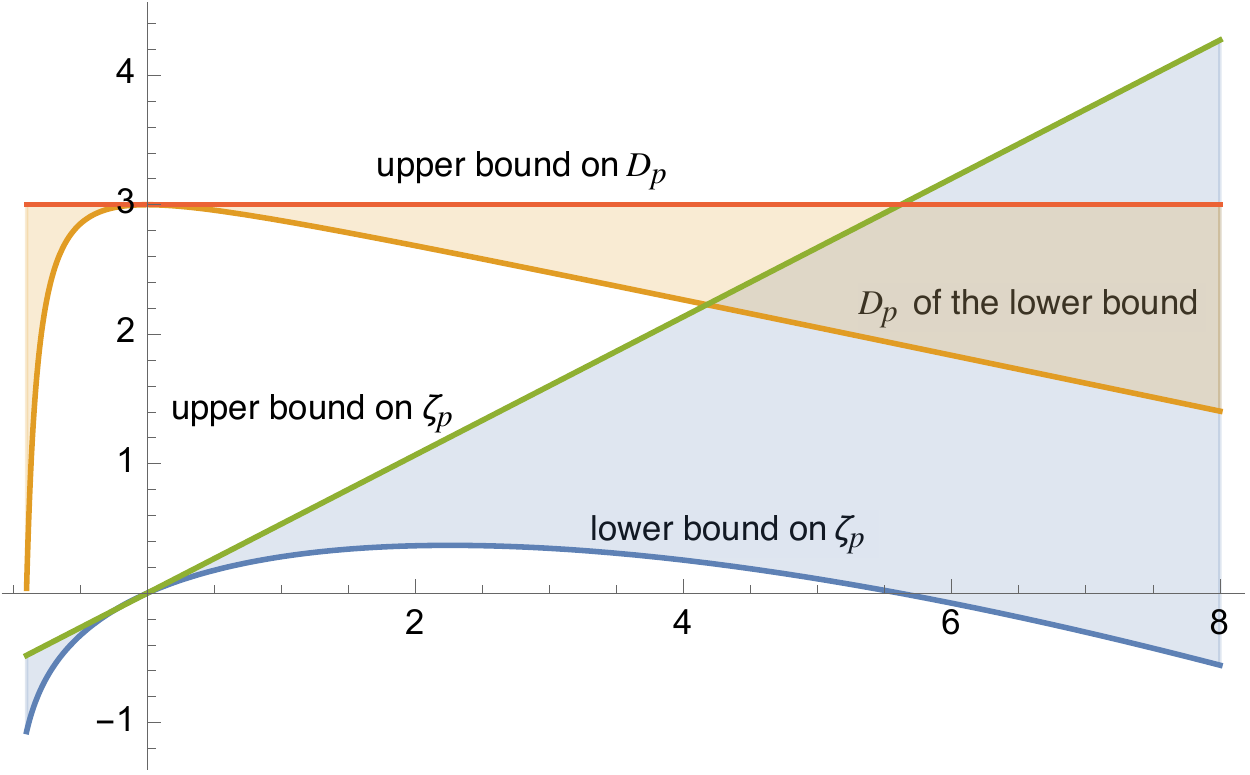}
	\caption{Graphs of the upper and lower bounds on $\zeta_p$ \eqref{e:zetaplower2}, upper bound on $D_p$, $D_p \leq 3$, and $D_p$ computed for the lower bound on $\zeta_p$ \eqref{eq:entreme_D_p} with $\ell=0.1$.}
	\label{f:graph}
\end{figure}

\section{Volume factors, active velocities and regions} \label{s:active}

This section is devoted to analytical scrutiny of the newly introduced concepts.  The main purpose here will be to demonstrate that the volume factors, dimensions and the associated thresholds define physically meaningful notions that   extract the right information from the field $\bu$ relevant in multifractal turbulence. 

The actual origins of the field will not play a role here. We will simply fix a function $f$ which belongs to a range of $L^p$-spaces defined over a probability measure space $(\Omega,\mu)$. The averages $\lan \cdot \ran$ mean the usual integration over $\O$:
\begin{equation}\label{}
\lan f \ran = \int_{\O} f(\o) \dmu(\o).
\end{equation}

\subsection{Volume factors and dimensions} \label{sec:active-volumes-intro} 
We start with formalities of the active volume/region theory.  Let $-\infty \leq p,q \leq \infty$, $q\neq p$. We define  the \emph{$(q,p)$-volume factor} of $f$ as follows
\begin{equation}\label{e:vol-intro}
V_{q,p} = \frac{ \lan |f|^q \ran^{\frac{p}{p-q}}}{ \lan |f|^p \ran^{\frac{q}{p-q}}}.
\end{equation}
Here are a few properties of volume factors that are easy to verify: 
\begin{itemize}
    \item[(V1)] $V_{q,p}$ is adimensional;
    \item[(V2)] symmetric:  $V_{q,p} = V_{p,q}$;
    \item[(V3)] homogeneous: $V_{q,p}(f) = V_{p,q}( \l f)$;
    \item[(V4)] $V_{q,p} \leq 1$ if $pq \geq 0$, and $V_{q,p} \geq 1$ if $pq \leq 0$. Also, $V_{q,0}=V_{0,p} = 1$.
    \item[(V5)] log-convex:  for any triple $p_1<p_2<p_3$, and any $q$, 
    \begin{equation}\label{e:logconv}
\begin{split}
 V_{q,p_2} &\leq (V_{q,p_1})^{\frac{q-p_1}{q-p_2} \frac{p_3-p_2}{p_3-p_1}}(V_{q,p_3})^{\frac{q-p_3}{q-p_2} \frac{p_2-p_1}{p_3-p_1}}, \quad q>p_2 \\
  V_{q,p_2} &\geq (V_{q,p_1})^{\frac{q-p_1}{q-p_2} \frac{p_3-p_2}{p_3-p_1}}(V_{q,p_3})^{\frac{q-p_3}{q-p_2} \frac{p_2-p_1}{p_3-p_1}}, \quad q<p_2.
\end{split}        
    \end{equation}
\end{itemize}

Let us simply notice that (V5) is a consequence of the interpolation inequality,
\begin{equation}
 \lan |f|^{p_2} \ran   \leq    \lan |f|^{p_1} \ran^\frac{p_3-p_2}{p_3-p_1}     \lan |f|^{p_3} \ran^\frac{p_2-p_1}{p_3-p_1}  .
\end{equation}

Slightly less trivial is the following property.

\begin{lemma}\label{l:V6}
The volume-factors obey the following monotonicity properties:
\begin{equation}\label{e:V6}
\p_q V_{q,p}\leq 0  \text{ if } p>0, \text{ and  } \p_q V_{q,p}\geq 0 \text{ if } p\leq 0. \tag{V6}
\end{equation}
\end{lemma}
\begin{proof}
If $q>p$, then we use the following form of $V_{q,p}$:
\[
V_{q,p} =  \lan |f|^p \ran  \left(  \frac{1}{ \left\langle \frac{|f|^p}{\lan |f|^p \ran  } |f|^{q-p} \right\rangle^{\frac{1}{q-p} }  }\right)^p
\]
Clearly the expression on the bottom represents a $L^{q-p}$ norm with respect to the normalized measure $\frac{|f|^p}{\lan |f|^p \ran  }\dmu$. Hence, the expression is increasing if $p>0$ and decreasing if $p\leq 0$.  For $q<p$ we first observe that 
\[
V^\e_{q,p} = \frac{ \lan |f|^q \chi_{|f| >\e} \ran^{\frac{p}{p-q}}}{ \lan |f|^p \ran^{\frac{q}{p-q}}} \to V_{q,p}.
\]
But
\[
V^\e_{q,p} =  \lan |f|^p \ran  \left(  \left\langle \frac{|f|^p}{\lan |f|^p \ran  } \left( \frac{\chi_{|f| >\e}}{|f|}\right) ^{p-q}  \right\rangle^{\frac{1}{p-q}  }\right)^p.
\]
By the same token $V^\e_{q,p}$ is decreasing in $q$ if $p>0$ and increasing if $p<0$, and consequently so is $V_{q,p} $. This proves \eqref{e:V6}. 
\end{proof} 

We note that by the symmetry we have similar monotonicity in $p$. This makes $(0,0)$ a saddle point for the volume factors.

For a fixed scale $\ell <1$, we also define the corresponding two-parameter family of dimensions
\begin{equation} \label{e:Dpq3}
    D_{q,p} = 3 - \log_\ell V_{q,p}.
\end{equation}
The log-convexity of expressed in \eqref{e:logconv} translates into the conventional convexity for the dimensions:
\begin{equation}
\begin{split}
 D_{q,p_2} &\leq \l_1 D_{q,p_1} + \l_2 D_{q,p_3}, \quad q>p_2 \\
  D_{q,p_2} &\geq \l_1 D_{q,p_1} + \l_2 D_{q,p_3}, \quad q<p_2 \\
 \l_1 = \frac{q-p_1}{q-p_2} \frac{p_3-p_2}{p_3-p_1} , &\quad \l_2 = \frac{q-p_3}{q-p_2} \frac{p_2-p_1}{p_3-p_1}.
\end{split}
\end{equation}
We recover the same monotonicity for $D_{q,p}$ as for volumes (V6), and from (V4) we obtain
\begin{equation}\label{e:D3}
D_{0,p}=D_{q,0} = 3, \quad D_{q,p}\leq 3 \text{ if } pq\geq 0,  D_{q,p}\geq 3 \text{ if } pq\leq 0.
\end{equation}

\subsection{Concentration of information}
The basic analytical meaning of a volume factor is to give a measure of a set containing much of the information carried by the source $f$. More precisely, we have the following lemma.

\begin{lemma}\label{l:AV} Let $0\leq p < q \leq \infty$. 	There exists a set $A \ss \O$ with $\mu(A) = V_{q,p}$ such that 
	\begin{equation}\label{}
	(1-c_{q,p}) \int_{\O} |f|^q \dmu \leq \int_{A} |f|^q \dmu, 
	\end{equation}
	where $c_{q,p} = \left( \frac{q-p}{q} \right)^{\frac{q-p}{q}}\left( \frac{p}{q} \right)^{\frac{p}{q}}$.
\end{lemma}
\begin{proof} If $V_{q,p}=1$, the statement is trivial. Suppose $V_{q,p}<1$. Note that the function $\mu( \{ |f| \geq \a\})$ is continuous from the left, and at a point of a jump the size of the jump is exactly $\mu( \{ |f| = \a\})$.  Hence, there exists an $\a \geq 0$  and a set $B \ss \{ |f| = \a\}$ such that $A = \{ |f| >\a \} \cup B$ has measure exactly $V_{q,p}$.  By Chebyshev's inequality,
	\[
	V_{q,p} \leq \frac{1}{\a^p}\int_A |f|^p d\mu.
	\]
	Using this and the fact that $|f| \leq \a$ on $\O \bs A$ we obtain
	\[
	\lan |f|^q \chi_{\O \bs A} \ran \leq \a^{q-p} \lan |f|^p \chi_{\O \bs A} \ran \leq \frac{1}{V_{q,p}^{\frac{q-p}{p}}} \lan |f|^p \chi_{A} \ran^{\frac{q-p}{p}} \lan |f|^p \chi_{\O \bs A} \ran \leq \lan |f|^q \ran  \frac{ \lan |f|^p \chi_{\O \bs A} \ran^{\frac{p}{q}} \lan |f|^p \chi_{A} \ran^{\frac{q-p}{p} }}{ \lan |f|^p \ran }
	\]
	Note that the latter fraction is of the form $ \th^{\frac{p}{q}} (1-\th)^{\frac{q-p}{p}}$, $\th \in [0,1]$, which attains its maximum exactly at  the value $c_{q,p}$. This proves the lemma.
\end{proof}

Now there are several natural questions to ask. 

First, for a fixed $q$, what is the most optimal exponent $p$ that recovers most of the function? In other words, what is the smallest constant $c_{q,p}$? Elementary optimization yields, $c_{min} = \frac{1}{2}$, which is achieved when $p = \frac{q}{2}$.  So, it implies that the volume factors will recover a half of the $q$-weight of the function at best, provided we choose $p = q/2$: there exists a set $A \ss \O$ with $\mu = V_{q,\frac{q}{2}}$ such that 
\begin{equation}\label{}
	\frac12 \int_{\O} |f|^q \dmu \leq \int_{A} |f|^q \dmu.
	\end{equation}
It is worth noting that the classical flatness factor 
\[
\cF = V^{-1}_{4,2} = \frac{ \lan |f|^4 \ran }{ \lan |f|^2 \ran^2}
\]
emerges in this context as a particular case of an optimal concentration factor.

Second, we can see that on the other end letting $p\ra q$ gives $c_{q,p} \ra 1$. As a result, in this limit the information about concentration of the function $f$ gets lost. However, the corresponding volume factors do not degenerate, and converge to something non-trivial, namely,
\begin{equation}\label{}
V_p = \lim_{q \ra p} V_{q,p} = \lan |f|^p \ran \exp\left\{ - \frac{\lan |f|^p \ln |f|^p \ran}{\lan |f|^p \ran}\right\} = \exp\left\{ - \frac{\lan |f|^p \ln \frac{|f|^p}{ \lan |f|^p \ran} \ran}{\lan |f|^p \ran}\right\}.
\end{equation}
Let us call them {\em $p$-volume factors}. The corresponding dimensional values defined by
\begin{equation}\label{e:Dpf}
D_p := 3 - \log_\ell V_{p} = 3 - \log_\ell \lan |f|^p \ran  +  \frac{\lan |f|^p \log_\ell |f|^p \ran}{\lan |f|^p \ran}.
\end{equation}
have played the central role in the multifractal formalism we described in the previous section.

So, there arises a natural question:  what kind of concentration not captured by \lem{l:AV} do these factors represent? The short answer -- it is the concentration of the entropy of the renormalized distribution $F =   \frac{|f|^p}{ \lan |f|^p \ran}$.

To start, let us define the entropy $H =  \lan F \ln F \ran $, and note the following simple formula
\begin{equation}\label{e:VF}
V_p = e^{-H}.
\end{equation}

\begin{remark} As a side remark we note that $V_p$  is directly related to the classical  Shannon information: $\cN(F) = \frac{1}{2\pi e} \exp\left\{ - \frac{2}{3} H\right\}$ via the power law:
\[
\frac{1}{2\pi e} V_p^{\frac{2}{3}} (f) = \cN(F) ,
\]
see \cite{Villani} for more details.
\end{remark}

The entropy $H$ itself measures how close function $f$ is to a constant via the classical Csisz\'ar-Kullback inequality, see \cite{Villani}:
\[
\frac12 \lan |F - 1| \ran ^2 \leq \lan F \ln F \ran \leq \lan |F- 1|^2 \ran.
\]
It terms of $f$ itself we obtain
\begin{equation}
\frac12 \left\langle  \left|  \frac{|f|^p}{ \lan |f|^p \ran} - 1 \right| \right\rangle ^2 \leq  - \ln V_p \leq \left\langle  \left|  \frac{|f|^p}{ \lan |f|^p \ran} - 1 \right|^2 \right\rangle.
\end{equation}
So, the closer $V_p$ is to $1$, the less concentrated the function $f$ is, i.e. the more uniform it becomes.  

Closer to the smaller range, however, the $p$-factors regulate concentration of the entropic density $F\ln F$. We will see this expressed in the following two lemmas that establish concentration in two different ways. First, concentration in a weak sense states that we have to increase the allowed volume slightly to $V^{1-\e}$ in order to achieve concentration up to the proportion $\e$. Although weak, this result works under no  restrictions on $H$. And second, strong concentration states that $F\ln F$ indeed concentrates on a volume  $V_p$ but under an upper cap on the size of the entropies, $H  \leq H_0$. 

Both results hold for a general probability density $F$, its entropy $H = \lan F \ln F \ran$, and volumetric factor $V = e^{-H}$.

\begin{lemma}[Weak concentration]\label{l:weakconc} For any probability distribution $F$, and any $0<\e<1$ there exists a set $A \ss \O$ with $\mu(A) = V^{1-\e}$ such that 
    \begin{equation}\label{e:Vc}
    \e H \leq \int_{A} F \ln F \dmu.
    \end{equation}
\end{lemma}
\begin{proof}
   We start as in the proof of \lem{l:AV} and find $A \ss \{F \geq \a\}$ such that $\mu(A) = V^{1-\e}$.  Then
  \[
  \lan F \ln F \chi_{\O \bs A} \ran \leq \lan F \chi_{\O \bs A} \ran \ln \frac{\lan F \chi_A \ran}{V^{1-\e}}    = \lan F \chi_{\O \bs A} \ran ( \ln \lan F \chi_A \ran + (1-\e)\lan F \ln F \ran ).
  \]
 We have $H =  \lan F \ln F \ran$, and hence the above is bounded by the maximum of the function
 \[
 g(x) = x \ln (1-x) + (1-\e)Hx, \quad x\in [0,1],
 \]
 which is bounded by $(1-\e) H$. 
\end{proof}
In terms of $f$, the result reads as follows
 \begin{equation}\label{e:Vc}
	\e \int_{\O} |f|^p \ln \frac{|f|^p}{ \lan |f|^p \ran} \dmu \leq \int_{A} |f|^p \ln \frac{|f|^p}{ \lan |f|^p \ran} \dmu.
\end{equation}

The same proof in fact implies the strong version and allows us to reach the original volume $\mu(A) = V$ provided we assume that the entropy has a priori known bound $H\leq H_0$.  Then, we obtain  concentration with the constant $c_{H_0}$ given by
\[
c_{H_0} = \sup_{x\in[0,1], H\leq H_0} \frac{1}{H} x \ln (1-x) + x.
\]
\begin{lemma}[Strong concentration]\label{l:strongconc} For any $H_0>0$ and probability distribution $F$, with $H\leq H_0$  there exists a set $A \ss \O$ with $\mu(A) = V$ such that 
	\begin{equation}\label{e:Vc}
		(1-c_{H_0}) H \leq \int_{A} F \ln F  \dmu.
	\end{equation}
\end{lemma}
Unfortunately in this case we lose information  $c_{H_0} \to 1$, as $H_0 \to \infty$. An example showing that this worsening is indeed happening can be constructed as follows.  Let $\{I_i\}_{i=1}^\infty$ be a family of disjoint intervals on $[0,1]$ with $|I_i| = \frac{1}{2^i}$.  Define $F = \sum_{i=1}^n \frac{2^i}{n} \chi_{I_i}$.  Then $\lan F \ran =1$ and 
\[
H = \frac{1}{n} \sum_{i=1}^n \ln\frac{2^i}{n} \sim n \ln2 - \ln(n),
\]
and $V \sim \frac{n}{2^n}$. In order to capture a set of this volume we need to take last $n-k$ intervals so that 
\[
|\{F \geq \a \}| = \sum_{i=k}^n \frac{1}{2^i} \sim \frac{1}{2^k} =  \frac{n}{2^n}.
\]
This implies $k = n - \log_2(n)$.  But  then
\[
\int_A F \ln F \dx = \frac{1}{n} \sum_{i=k}^n \ln\frac{2^i}{n} = \ln \frac{2^{\frac{k+...+n}{n}}}{n^{\frac{n-k}{n}}} \lesssim \ln n,
\]
which is an order smaller than $H$.

In summary, we have demonstrated that the volume factors represent a measure of concentration of the function $f$ or its renormalized entropy $F\ln F$ in the case of 1-parameter factors.  This motivates us to look further into the question of what hight levels of the function $f$, which already appeared in the proofs of this section,  determine the threshold for its ``most active" values. 

\subsection{Active thresholds and regions}

Lemmas~\ref{l:AV}, \ref{l:weakconc}, \ref{l:strongconc} provide little constructive information about the threshold $\a$ that define concentration sets $A$. In this section we will give a more physically relevant concept of an appropriate level $\a$, called active threshold, which is similar in spirit to the volume factors. The upside of this approach is that such a definition will be computationally accessible, as we give a precise formula for $\a$. A  downside is that the active concentration set $A$ will not have the exact same measure as $V_{q,p}$, but rather its constant multiple. 

So, let us define an \emph{$(q,p)$-active threshold} by
\begin{equation}\label{}
s_{q,p} = \frac{ \lan |f|^q \ran^{\frac{1}{q-p}}}{ \lan |f|^p \ran^{\frac{1}{q-p}}}.
\end{equation}
We recall that in the context of turbulence, $f = \d_\ell \bu$, and therefore $s_{q,p}$ depend on the scale. In \cite{CS2014} these appeared in the special case of $p=2$, $q = 3$, under the term of active speeds which refers to their physical unit. 

By analogy with the previous section let us list some of the fundamental properties of active thresholds:
\begin{itemize}
	\item[(s1)] unit of $s_{q,p}$ is the same as the unit of  $f$;
	\item[(s2)] symmetric:  $s_{q,p} = s_{p,q}$;
	\item[(s3)] $1$-homogeneous: $s_{q,p}(\l f) = \l s_{q,p}(f)$;
	\item[(s4)] bound from above: $s_{q,p} \leq \|f\|_\infty$;
	\item[(s5)] log-convexity:  for any triple $p_1<p_2<p_3$, and any $q$, 
	\begin{equation}\label{e:sconvexity}
	\begin{split}
	s_{q,p_2} &\leq (s_{q,p_1})^{\frac{q-p_1}{q-p_2} \frac{p_3-p_2}{p_3-p_1}}(s_{q,p_3})^{\frac{q-p_3}{q-p_2} \frac{p_2-p_1}{p_3-p_1}}, \quad q<p_2 \\
	s_{q,p_2} &\geq (s_{q,p_1})^{\frac{q-p_1}{q-p_2} \frac{p_3-p_2}{p_3-p_1}}(s_{q,p_3})^{\frac{q-p_3}{q-p_2} \frac{p_2-p_1}{p_3-p_1}}, \quad q>p_2.
	\end{split}        
	\end{equation}
	\item[(s6)] monotonicity:  $\p_q s_{q,p}\geq 0$ for any $p$.
	\item[(s7)] volume-threshold-structure function relation: 
	\[
	s_{q,p}^p V_{q,p} = \lan |f|^p \ran.
	\]
\end{itemize}
Note that (s6) is a consequence of (s7) and (V6). 

Now to show relevance of the newly introduced thesholds to the concentration phenomena let us fix a family of  adimensional constants $0<\s_{q,p}<1$, and define the  \emph{$(q,p)$-active region} by
\begin{equation}\label{}
A_{q,p} = \{ |f| \geq \s_{q,p} s_{q,p} \}.
\end{equation}
By Chebyshev's inequality, for $p>0$, we readily obtain the bound
\[
|A_{q,p}| \leq \frac{1}{\s_{q,p}^p s^p_{q,p}} \lan |f|^p \ran = \frac{1}{\s_{q,p}^p} V_{q,p},
\]
and in the case $p<q$,
\[
\lan |f|^q \chi_{\O \bs A_{q,p} } \ran \leq \s_{q,p}^{q-p} \frac{ \lan |f|^q \ran}{ \lan |f|^p \ran} \lan |f|^p \chi_{\O \bs A_{q,p} } \ran \leq  \s_{q,p}^{q-p} \lan |f|^q \ran,
\]
which implies the expected concentration property:
\begin{equation}\label{}
(1-\s_{q,p}^{q-p}) \lan |f|^q \ran \leq   \lan |f|^q \chi_{A_{q,p}} \ran .
\end{equation}
Due to this inequality it is natural to set $\s_{q,p} = c_{q,p}^{\frac{1}{q-p}}$, which corresponds to the same level of concentration as the superlevel sets $\{|f|\geq \a\}$ in \lem{l:AV}. Although such a choice of parameters may not be well-justified physically, it is certainly resonates mathematically with the prior result.

We now investigate the relevance of these active regions to the formalism of H\"older  sets $A_h$ we discussed in \eqref{e:Ah}. To this end, we let  $q\ra p$. Once again, we observe the same phenomenon -- the information about concentration of $f$ gets lost since $\s_{q,p}^{q-p} \to 1$. However, the active thresholds converge to something non-trivial, namely
\begin{equation}\label{e:spdef}
\lim_{q \ra p} s_{q,p} = s_p=  \exp\left\{  \frac{\lan |f|^p \ln |f| \ran}{\lan |f|^p \ran}\right\}.
\end{equation}

A remarkable property of this 1-paramter family is that it stores information about the entire original 2-parameter family via the following restoration formula \begin{equation}\label{e:means}
s_{q,p} = \exp\left\{ \frac{1}{q-p} \int_p^q \ln s_{r}\, \dr \right\}.
\end{equation}
Concerning other properties the reader can readily check that  $s_p$'s inherit the same (s1)--(s4), (s6), while (s7) translates into 
\begin{equation}\label{e:sSV}
s_{p}^p V_{p} = \lan |f|^p \ran.
\end{equation}

Let us now go back to the physical interpretation of $s_p$'s. First, let us observe that for $f = \d_\ell \bu$, $s_p = \ell^h$ are exact same values that appeared in  \eqref{e:spfirst} in our initial heuristic argument, which define the active regions $A_h$ of H\"older regularity. There $h$ is related to $p$ via the usual transformation $h = \zeta'_p$. So,  if we define $A_p$  in the new terms parameterized by $p$
\begin{equation}\label{ }
A_p = \{ c s_p \leq |\d_\ell \bu| \leq C s_p\},
\end{equation}
then as before by Chebyshev's inequality and \eqref{e:sSV},
\[
\mu(A_p) \lesssim \frac{\lan |\d_\ell \bu|^p \ran}{s_p^p} = V_p = \ell^{3- D_p}.
\]
This provides the direct analogue of \eqref{e:Ahdh} in terms of the volumetric quantities.

As in the last section there remains one last unsettled issue: since in the limit as $q \to p$ the concentration information of the sets $A_{q,p}$ deteriorates, we still would like to understand what kind of concentration do the limiting thresholds $s_p$ capture? As previously we will find that it is responsible for concentration of entropic densities.

To that end, we extract the physical unit of $s_p$, equal that of $f$, upfront. Namely, if the physical unit of $f$ is $U_0$ we rewrite the formula for $s_p$ as follows:
\begin{equation}
s_p= U_0 \exp\left\{  \frac{\lan |f|^p \ln \frac{|f|}{U_0} \ran}{\lan |f|^p \ran}\right\}.
\end{equation}
Note that this is the exact same formula as appeared the original definition \eqref{e:spdef}, and it holds for any $U_0$. We can view the ratio 
\[
\frac{s_p}{U_0} = \exp\left\{  \frac{\lan |f|^p \ln \frac{|f|}{U_0} \ran}{\lan |f|^p \ran}\right\}.
\]
as an adimensional threshold value for an active velocity. 

\begin{lemma}\label{} Let  
\begin{equation}
A = \{ |f| \geq  U_0^{1-c} s_p^{c} \},
\end{equation}
where $0<c<1$. The function $f$ concentrates on $A$ in the following sense
\begin{equation}\label{e:pcons}
(1-c) \left\langle  |f|^p  \ln_+ \frac{|f|}{U_0} \right\rangle  \leq \left\langle  |f|^p  \ln_+ \frac{|f|}{U_0} \, \chi_{A} \right\rangle .
\end{equation}  
\end{lemma}
Here we denote
\[
\ln_+ = \begin{cases} \ln,& \text{ if } \ln>0 \\ 0,&  \text{ if } \ln=0 \end{cases}
\]
\begin{proof}
To see \eqref{e:pcons} note that on $A$ we have 
\[
\frac{|f|}{U_0} \geq \left( \frac{s_p}{U_0} \right)^{c}.
\]
Hence, on the complement we have
\[
\left\langle  |f|^p  \ln_+ \frac{|f|}{U_0} \, \chi_{\O_T \backslash A} \right\rangle \leq c \left\langle  |f|^p  \frac{ \lan |f|^p \ln \frac{|f|}{U_0} \ran_+  }{ \lan |f|^p \ran} \right\rangle =  c \left\langle  |f|^p \ln \frac{|f|}{U_0} \right\rangle_+ \leq c \left\langle  |f|^p \ln_+ \frac{|f|}{U_0} \right\rangle.
\]
This implies \eqref{e:pcons}. 
\end{proof}

\section{Energy spectrum} The purpose of this section is to provide a rigorous link between scaling properties of the second order structure function and the energy spectrum. To make it rigorous we define the energy spectrum  classically as the Fourier transform of the properly defined correlation function.

So, let $\bu$ be a velocity field. We define a 3D correlation function as follows
\begin{equation}
\G(\by) = \frac12 \int_{O(3)}  \lan \bu(\cdot + A \by) \cdot \bu(\cdot) \ran  \dnu(A),
\end{equation}
where $O(3)$ is the orthogonal group on $\R^3$, and $\dnu$ is the normalized Haar measure on it. 
Clearly, $\G(0) = \cE = \frac12 \langle |\bu|^2 \rangle$ is the twice total energy.  We will work under assumption that the correlation function decays sufficiently fast at infinity. It suffices to have 
\begin{equation}\label{e:Gdecay}
    |\n^k \G(\by)| \lesssim \frac{1}{1+|\by|^{k-1}}, \quad k \leq 5.
\end{equation}
We define the 3D  energy spectrum by
\begin{equation}\label{e:E}
E_{3D}(\bkap) = \int_{\R^3} e^{-i \bkap \cdot \by} \G(\by) \dby.
\end{equation}
Using Parseval's identity we obtain
\[
E_{3D}(\bkap) = \frac12 \frac{1}{(2\pi)^3} \int_{\R^3} e^{-i \bkap \cdot \by}\int_{O(3)} |\hat{\bu}(\bxi)|^2 e^{i A \by \cdot \bxi} \dbxi \dnu(A) \dby.
\]
Moving the $\dby$-integral inside yields a Dirac at $A^{-1} \bxi = \bkap$. Hence,
\[
E_{3D}(\bkap) = \frac12 \int_{O(3)} |\hat{\bu}(A\bkap)|^2 \dnu(A).
\]
It is a classical result that the push-forward of the Haar measure $\dnu$ through the map $T_\bkap:A \to A\bkap$ is the normalized surface measure on the sphere $\k \S^2$, $\k = |\bkap|$.  We thus obtain
\[
E_{3D}(\bkap) = \frac{1}{8 \pi \k^2} \int_{\k \S^2}  |\hat{\bu}(\bxi)|^2 \dbxi.
\]
This leads to the classical energy spectrum definition
\[
E(\k) = 4\pi \k^2 E_{3D}(\bkap) = \frac12 \int_{\k \S^2}  |\hat{\bu}(\bxi)|^2 \dbxi.
\]
Thus, the total energy is recovered from the spectrum by
\[
\int_0^\infty E(\k) \dk = \cE.
\]
We now relate $E(\k)$ and the second order structure function.  Note that
\[
\G(\by) = \cE + \frac12 \int_{O(3)}  \lan (\bu(\cdot + A \by) - \bu(\cdot)) \cdot \bu(\cdot) \ran  \dnu(A).
\]
Shifting inside the average by $A\by$ and reversing the sign of $\by$ we obtain
\[
\G(\by) = \cE - \frac14 \int_{O(3)}  \lan |\bu(\cdot + A \by) - \bu(\cdot)|^2 \ran  \dnu(A).
\]
The obtained average is the second order structure function 
\[
S_2(\ell) =  \frac14 \int_{O(3)}  \lan |\bu(\cdot + A \by) - \bu(\cdot)|^2 \ran  \dnu(A), \quad \ell = |\by|.
\]
One can now relate scaling laws of $S_2(\ell)$ to those of the energy spectrum.

\begin{proposition}\label{p:spec}
	Suppose $S_2(\ell) = c \ell^\a$ for all $\ell < \ell_0$ and some $0<\a<2$.  Then 
    \begin{equation}
E(\k) = \frac{c(\k)}{\k^{1+\a}}, \quad \text{ where } c(\k) \to c_0 >0 \text{ as } \k \to \infty.
    \end{equation}
Conversely, if  $E(\k) = c \k^{-1-\a}$ for all $\k \geq \k_0$, then 
\begin{equation}
S_2(\ell) = c(\ell) \ell^\a, \quad \text{ where } c(\ell) \to c_0 >0 \text{ as } \ell \to 0.
\end{equation}
\end{proposition}
We note that the result actually holds for all $\a>0$ with respectively stronger assumptions on the decay of the correlation function. We restrict ourselves to $\a<2$ since in the extreme intermittency event we only have $\a = \frac53$.  For the Kolmogorov regime we have $S_2(\ell) = (\e \ell)^{2/3}$, hence $E(\k) \sim \k^{-5/3}$.
\begin{proof}
  Using that $\Gamma$ is a radial function, one can compute the Fourier integral of $\Gamma$ as a function of scalar $\ell$ as follows
 \[
 E(\k) =  c \int_0^\infty \k \ell \sin(\k \ell) \Gamma(\ell) \dell.
 \]
Let us assume w.l.o.g. that $\ell_0 = 1$. Let  $0\leq \chi \leq 1$ be a  smooth cutoff  function with $\supp \chi \ss \{0<\ell<1\}$ and $\chi(\ell)=1$ for $\ell \leq \frac12$. We can write
\[
E(\k) = c \int_0^\infty \k \ell \sin(\k \ell) \Gamma(\ell) \chi(\ell)\dell + R(\k),
\]
where 
\[
R(\k) = c \int_0^\infty \k \ell \sin(\k \ell) \Gamma(\ell) [1-\chi(\ell)]\dell .
\]
Using that $\frac{1}{\k^{4}} \frac{d^{4}}{d\ell^{4}} \sin(k\ell) =  \sin(k\ell)$, and integrating by parts we obtain
\[
R(\k) = \frac{c}{\k^{3}} \int_0^\infty \sin(\k \ell) [\ell \Gamma(\ell) (1-\chi(\ell)) ]^{(4)}\dell .
\]
Note that the under our assumption \eqref{e:Gdecay}, the integral converges absolutely, and thus $|R(\k)| \lesssim \frac{c}{\k^{3}}$, which is of smaller order.   Continuing with the main term we first split
\[
\int_0^1 \k \ell \sin(\k \ell) \Gamma(\ell) \chi(\ell)\dell  = \cE \int_0^1 \k \ell \sin(\k \ell) \chi(\ell)\dell  -  \k  \int_0^1 \sin(\k \ell) \ell^{1+\a} \chi(\ell)\dell .
\]
The first integral decays as an arbitrary power of $\k$ which can be seen by performing similar computation as above.  For the second we integrate by parts twice:
\begin{equation}\label{e:a01}
-  \k  \int_0^1 \sin(\k \ell) \ell^{1+\a} \chi(\ell)\dell  = - \int_0^1 \cos(\k \ell) ( \ell^{1+\a} \chi(\ell)  )' \dell   = \frac{1}{\k}  \int_0^1 \sin(\k \ell) ( \ell^{1+\a} \chi(\ell)  )'' \dell .
\end{equation}
Let us consider the case $0<\a<1$ at this point.  Expanding 
\[
( \ell^{1+\a} \chi(\ell)  )''  = (1+\a) \a \ell^{-1+\a} \chi(\ell)+ 2 (1+\a) \ell^{\a} \chi'(\ell) + \ell^{1+\a} \chi''(\ell)
\]
we can see that the latter two terms are supported away from the origin.  We can apply the same integration by parts as before to show that those two terms decay as any power of $\k$.  Finally, 
\[
\frac{1}{\k}  \int_0^1 \sin(\k \ell)  \ell^{-1+\a} \chi(\ell)  \dell = \frac{1}{\k^{1+\a}}  \int_0^\k \sin(x)  x^{-1+\a} \chi(\ell/\k)  \dell.
\]
Denoting $c(\k) =  \int_0^\k \sin(x)  x^{-1+\a} \chi(\ell/\k)  \dx$, due to condition $0<\a<1$ we can see that the integral results in a convergent alternating series with vanishing error term. Moreover,  $\lim_{\k \to \infty} c(\k) =  \int_0^\infty \sin(x)  x^{-1+\a}  \dx > 0$. This finishes the case $0<\a<1$.

In the case $\a =1$, we integrate by parts again in \eqref{e:a01}:
\[
\frac{1}{\k}  \int_0^1 \sin(\k \ell) ( \ell^{2} \chi(\ell)  )'' \dell =  \frac{1}{k^2} + \frac{1}{k^2} \int_0^1 \cos(\k \ell) ( \ell^{2} \chi(\ell)  )''' \dell .
\]
Now the integrand is supported away from the origin and hence decays at least as $1/k^3$. 

In the case $1<\a<2$ we obtain
\[
\frac{1}{\k}  \int_0^1 \sin(\k \ell) ( \ell^{1+\a} \chi(\ell)  )'' \dell =  \frac{1}{k^2} \int_0^1 \cos(\k \ell) ( \ell^{1+\a} \chi(\ell)  )''' \dell.
\]
Here the main term is 
\[
 \int_0^1 \cos(\k \ell) \ell^{-2+\a} \chi(\ell)  \dell \sim \frac{1}{\k^{1+\a}}.
 \]
  This proves the first part. 

The converse statement follows similarly by taking inverse Fourier transform of \eqref{e:E}.

\end{proof}

\newcommand{\etalchar}[1]{$^{#1}$}

%\bibliographystyle{alpha}
%\bibliography{turbulence}

\end{document}